\documentclass[a4paper,12pt, abstract=true]{scrartcl}

\usepackage{german}
\usepackage[ansinew]{inputenc}
\usepackage[german,english,french,USenglish]{babel}
\usepackage{latexsym,amssymb,amsfonts,bbm}
\usepackage{amsmath} 

\usepackage{theorem}
\usepackage{xcolor}
\usepackage{graphicx}
\usepackage{esint}
\usepackage{stmaryrd}

\newcommand{\eps}{\varepsilon}
\xdefinecolor{dgreen}{rgb}{0,0.6,0}
\xdefinecolor{dred}{rgb}{0.6,0,0}

\usepackage{verbatim}

\newtheorem{dummytheorem}{Dummy-Theorem}[section]
\newcommand{\proofendsign}{$\Box$} 
\newtheorem{definition}[dummytheorem]{Definition}
\newtheorem{lemma}[dummytheorem]{Lemma}
\newtheorem{theorem}[dummytheorem]{Theorem}

\newtheorem{corollary}[dummytheorem]{Corollary}

\newenvironment{proof}{{\noindent \bf Proof }}
 {{\hspace*{\fill}\proofendsign\par\bigskip}}

\theorembodyfont{\normalfont}
\newtheorem{remarknorm}[dummytheorem]{Remark}
\newtheorem{examplenorm}[dummytheorem]{Example}

\newcommand{\N}{\mathbb{N}}

\newcommand{\R}{\mathbb{R}}

\newcommand{\F}{\mathbb{F}}

\newcommand{\pr}{\mathbb{P}}
\newcommand{\ex}{\mathbb{E}}

\newcommand{\eins}{\mathbbm{1}}

\theorembodyfont{\rm}

\begin{document}


\title{On qualitative robustness of the Lotka--Nagaev estimator for the offspring mean of a supercritical Galton--Watson process}

\author{
Dominic Schuhmacher\footnote{University of G\"ottingen, Institute for Mathematical Stochastics; {\tt dschuhm1@uni-goettingen.de}}
\qquad
Anja Sturm\footnote{University of G\"ottingen, Institute for Mathematical Stochastics; {\tt asturm@math.uni-goettingen.de}}
\qquad
Henryk Zähle\footnote{Saarland University, Department of Mathematics; {\tt zaehle@math.uni-sb.de}}}
\date{}
\maketitle

\begin{abstract}
We characterize the sets of offspring laws on which the Lotka--Nagaev estimator for the mean of a supercritical Galton--Watson process is qualitatively robust. These are exactly the locally uniformly integrating sets of offspring laws,  which may be quite large. If the corresponding global property is assumed instead, we obtain uniform robustness as well. We illustrate both results with a number of concrete examples. As a by-product of the proof we obtain that the Lotka--Nagaev estimator is [locally] uniformly weakly consistent on the respective sets of offspring laws, conditionally on non-extinction.
\end{abstract}

\medskip


\medskip




\newpage

\section{Introduction}\label{Introduction}

A Galton--Watson branching process $(Z_n):=(Z_n)_{n\in\N_0}$ with initial state $1$ and offspring distribution $\mu$ on $\N_0:=\{0,1,2,\ldots\}$ describes the evolution of the size of a population with initial size $1$, where each individual $i$ in generation $k$ has a random number $X_{k,i}$ of descendants drawn from $\mu$ independently of all other individuals. 
In other words,
\begin{equation}\label{Def GWP}
    Z_0\,:=\,1\qquad\mbox{and}\qquad Z_n\,:=\,\sum_{i=1}^{Z_{n-1}}X_{n-1,i}\quad \mbox{for }n\in\N.
\end{equation}
For background see, for instance, \cite{AsmussenHering1983,AthreyaNey1972}. In this article we always assume that the mean
$$
    m_\mu\,:=\,\sum_{k=1}^\infty k\,\mu[\{k\}]
$$
of the offspring distribution $\mu$ is finite. A natural estimator for the offspring mean $m_\mu$ based on observations up to time $n$ is the Lotka--Nagaev estimator \cite{Lotka1939,Nagaev1967} given by
\begin{equation}\label{def of lotka estimator for mean of gwp}
    \widehat m_n
    \,:=\,
    \left\{
    \begin{array}{cll}
        \frac{\sum_{i=1}^{Z_{n-1}}X_{n-1,i}}{Z_{n-1}}=\frac{Z_n}{Z_{n-1}} & , & Z_{n-1}>0,\\
        0 & , & Z_{n-1}=0.
    \end{array}
    \right.
\end{equation}
This estimator requires knowledge only of the last two generation sizes $Z_{n-1}$ and $Z_n$. Another popular estimator is the Harris estimator $\sum_{k=1}^nZ_k/\sum_{k=0}^{n-1}Z_k$, which is known to be the nonparametric maximum likelihood estimator for $m_\mu$ when observing all generation sizes $Z_0,\ldots,Z_n$ \cite{Feigin1977,KeidingLauritzen1978} and even when observing the entire family tree \cite{Harris1948}. However, in this article we restrict ourselves to the Lotka--Nagaev estimator. Note that from the point of view of applications it is often the case that the process cannot be observed for an extended period of time, such that the Lotka--Nagaev estimator is the simplest or indeed the only possible choice in these situations.

In the critical and subcritical cases, i.e.\ when $m_\mu\le1$, the mean cannot be estimated consistently due to the extinction of $(Z_n)$ with probability $1$. On the other hand, in the supercritical case, i.e.\ when $m_\mu>1$, the Lotka--Nagaev estimator is strongly consistent on the set of non-extinction, which can be easily shown by adapting the argument of Heyde~\cite{Heyde1970}. Asymptotic normality (assuming finite variance of the offspring law $\mu$) on the set of non-extinction was obtained by Dion~\cite{Dion1974} among others. A discussion of further statistical properties can be found in \cite{DionKeiding1978}. For a recent overview of estimation in general branching processes we refer to \cite{MitovYanev2009}.

The objective of the present article is to investigate the estimator $\widehat m_n$ for (qualitative) robustness in the supercritical case. Informally, the sequence $(\widehat m_n)$ is robust when a small change in $\mu$ results only in a small change of the law of the estimator $\widehat m_n$ uniformly in $n$. More precisely, given a set ${\cal N}$ of probability measures $\mu$ on $\N_0$ with $m_\mu<\infty$, the sequence of estimators $(\widehat m_n)$ is said to be robust on ${\cal N}$ if for every $\mu_1\in{\cal N}$ and $\varepsilon>0$ there is some $\delta>0$ such that
\begin{equation}\label{def qual rob in intro}
    \mu_2\in{\cal N},~~d(\mu_1,\mu_2)\le\delta\quad\Longrightarrow\quad \rho(\mbox{\rm law}\{\widehat m_n|\mu_1\},\mbox{\rm law}\{\widehat m_n|\mu_2\})\le\varepsilon\quad\mbox{for all }n\in\N,
\end{equation}
where $d$ is any metric on ${\cal N}$ which generates the weak topology and $\rho$ is the Prohorov metric on the set ${\cal M}_1^+$ of all probability measures on $(\R_+,{\cal B}(\R_+))$. The sequence $(\widehat m_n)$ is said to be uniformly robust on ${\cal N}$ if $\delta$ can be chosen independently of $\mu_1\in{\cal N}$. [Uniform] robustness of $(\widehat m_n)$ on ${\cal N}$ means that the set of mappings $\{{\cal N}\rightarrow{\cal M}_1^+$, $\mu\mapsto\mbox{\rm law}\{\widehat m_n|\mu\} : n\in\N\}$
is [uniformly] $(d_{\scriptsize{\rm TV}},\rho)$-equicontinuous. This definition is in line with Hampel's definition of robustness for empirical estimators in nonparametric statistical models \cite{Cuevas1988,Hampel1971}. Note, however, that our situation is {\em not} covered by Hampel's setting, because our estimator $\widehat m_n$ is not based on $n$ i.i.d.\ observations. On the other hand, our setting is covered by the more general framework recently introduced in \cite{Zaehle2014b}. For background on robust statistics, see also \cite{Hampeletal1986,HuberRonchetti2009} and the references cited therein.

We point out that we do {\em not} claim that the Lotka--Nagaev estimator is particularly robust. For a ``robustification'' of the Lotka--Nagaev estimator, see  \cite{Stoimenovaetal2004}. We are rather interested in ``how robust'' the classical Lotka--Nagaev estimator is. To some extent, the degree of robustness of an estimator can be measured by the ``size'' of the sets ${\cal N}$ on which the estimator is robust; see also \cite{Zaehle2014b}. Intuitively, the larger the sets ${\cal N}$ on which the estimator is robust, the larger is the ``degree'' of robustness. Corollary \ref{asymptotic robustness of lotka - corollary} below gives an exact specification of these sets ${\cal N}$ for the Lotka--Nagaev estimator. Similar investigations have recently been done by Cont et al.\ \cite{Contetal2010} (see also \cite{Kraetschmeretal2014}) in the context of the empirical estimation of monetary risk measures. For instance, the empirical Value at Risk at level $\alpha$ (i.e., up to the sign, the empirical upper $\alpha$-quantile) is robust on the set ${\cal N}$ of all probability measures on $(\R,{\cal B}(\R))$ with a unique $\alpha$-quantile; cf.\ Proposition 3.5 in \cite{Contetal2010}.

Our main results state that the sets ${\cal N}$ on which the sequence $(\widehat m_n)$ is robust are exactly the locally uniformly integrating sets; and if a set ${\cal N}$ is even uniformly integrating and satisfies $\inf_{\mu\in{\cal N}} m_{\mu} > 1$, then $(\widehat m_n)$ is even uniformly robust on it. Uniformly integrating for a set ${\cal N}$ means just that any set of random variables $\{Y \sim \mu \colon \mu \in {\cal N} \}$ is uniformly integrable. This property is just a tiny bit stronger than finiteness of $\sup_{\mu \in {\cal N}} m_{\mu}$; see Remark~\ref{rem ui char}. Locally uniformly integrating means that every weakly convergent subsequence in ${\cal N}$ is uniformly integrating.

In Section 2 we also provide various examples of (parametric) sets ${\cal N}$ that are [locally] uniformly integrable. We illustrate the implied robustness statements in the context of estimating a parameter (via estimating the mean) that is either slightly perturbed or belongs to a model that is slightly misspecified. In both situations [uniform] robustness yields that the distribution of the estimator is largely unaffected.


\section{Main results and discussion}\label{Main results}

For the exact formulation of our main results we have to define the Galton--Watson process as a sort of canonical process. 
More precisely, let $(Z_n):=(Z_n)_{n\in\N_0}$ be given by (\ref{Def GWP}) with $(X_{k,i}):=(X_{k,i})_{(k,i)\in\N_0\times\N}$ the coordinate process on
$$
    (\Omega,{\cal F})\,:=\,(\N_0^{\N_0\times\N},\mathfrak{P}(\N_0)^{\otimes(\N_0\times\N)})
$$
(with $\mathfrak{P}$ denoting the set of all subsets) under the product law
$$
    \pr^\mu\,:=\,\mu^{\otimes(\N_0\times\N)}.
$$
Note that $(X_{k,i})$ is a double sequence of i.i.d.\ random variables with distribution $\mu$.

Let ${\cal N}_1^1$ be the set of all probability measures $\mu$ on $\N_0$ with $m_\mu<\infty$, and $d_{\scriptsize{\rm TV}}$ the total variation distance on ${\cal N}_1^1$, i.e.
\begin{equation}\label{def tv metric}
    d_{\scriptsize{\rm TV}}(\mu_1,\mu_2)
    \,:=\, \sup_{A\in\mathfrak{P}(\N_0)}\,|\mu_1(A)-\mu_2(A)|\\
    \,=\, \frac{1}{2}\sum_{k\in\N_0}\big|\mu_1[\{k\}]-\mu_2[\{k\}]\big|.
\end{equation}
As before let ${\cal M}_1^+$ be the set of all probability measures on $(\R_+,{\cal B}(\R_+))$ and $\rho$ be the Prohorov metric on ${\cal M}_1^+$, i.e.
\begin{equation}\label{def p metric}
    \rho(\mu_1,\mu_2)\,:=\,\inf\{\varepsilon>0  :\,\mu_1[A]\le\mu_2[A^\varepsilon]+\varepsilon\mbox{ for all }A\in{\cal B}(\R_+)\} 
\end{equation}
with $A^\varepsilon:=\{x\in\R_+:\,\inf_{a\in A}|x-a|\le\varepsilon\}$. Note that $d_{\scriptsize{\rm TV}}$ coincides with the Prohorov metric on ${\cal N}_1^1$. In particular, $d_{\scriptsize{\rm TV}}$ and $\rho$ metrize the weak topologies on ${\cal N}_1^1$ and ${\cal M}_1^+$, respectively.

\begin{definition}\label{def robustness}
For ${\cal N}\subset{\cal N}_1^1$, the sequence $(\widehat m_n)$ is said to be robust on ${\cal N}$ if for every $\mu_1\in{\cal N}$ and $\varepsilon>0$ there is a $\delta>0$ such that
$$
    \mu_2\in{\cal N},\quad d_{\scriptsize{\rm TV}}(\mu_1,\mu_2)\le\delta\quad\Longrightarrow\quad \rho(\pr^{\mu_1}\circ \widehat m_n^{-1}\,,\,\pr^{\mu_2}\circ \widehat m_n^{-1})\le\varepsilon \quad\mbox{for all }n\in\N.
$$
It is said to be uniformly robust on ${\cal N}$ if $\delta$ can be chosen independently of $\mu_1\in{\cal N}$.
\end{definition}

Of course, the notion of robustness remains the same when replacing $d_{\scriptsize{\rm TV}}$ by any other metric metrizing the weak topology. The main result of this article is Theorem~\ref{asymptotic robustness of lotka}. For its formulation we need a version of Definition 3.3 in \cite{Zaehle2014b} concerning locally uniformly $\psi$-integrating sets. Here, we set $\psi(k):=k$, $k\in\N_0$. Note that choosing the identity function for $\psi$ corresponds to the notion of locally uniformly integrating sets mentioned in the introduction. In our setting this choice is equivalent to considering $\psi_1$ when $\psi_p(k):=(1+k)^p, k\in\N_0, p \geq 0$ as introduced in (17) of \cite{Zaehle2014b}. This motivates the following definition and terminology.

\begin{definition}\label{def of uniformly integrating}
A set ${\cal N}\subset{\cal N}_1^1$ is said to be locally uniformly $\psi_1$-integrating if for every $\varepsilon>0$ and $\mu_1\in{\cal N}$ there exist some $\delta>0$ and $\ell\in\N$ such that
$$
    \mu_2\in{\cal N},\quad d_{\scriptsize{\rm TV}}(\mu_1,\mu_2)\le\delta\quad\Longrightarrow\quad \sum_{k=\ell}^\infty k\,\mu_2[\{k\}]\,\le\,\varepsilon.
$$
It is said to be uniformly $\psi_1$-integrating if for every $\varepsilon>0$ there exists some $\ell\in\N$ such that
$$
    \sup_{\mu\in{\cal N}}\,\sum_{k=\ell}^\infty k\,\mu[\{k\}]\,\le\,\varepsilon.
$$
\end{definition}

\begin{remarknorm}\label{rem ui char}
Any uniformly $\psi_1$-integrating set ${\cal N}$ is also locally uniformly $\psi_1$-integrating. We have the following characterizations of the two concepts.
\begin{enumerate}
  \item[(i)] It is straightforward to verify from the definition that a set ${\cal N}\subset{\cal N}_1^1$ is locally uniformly $\psi_1$-integrating if and only if every sequence $(\mu_n)\in{\cal N}^{\N}$ that converges weakly in ${\cal N}$ is uniformly $\psi_1$-integrating.
  \item[(ii)] The de la Vall\'ee-Poussin theorem (Theorem II.T22 in \cite{Meyer1966}) implies that a set ${\cal N}\subset{\cal N}_1^1$ is uniformly $\psi_1$-integrating if and only if there exists a sequence $(a_k)\in \R_+^{\N}$ such that $a_k/k\to\infty$ as $k\to\infty$ and $\sup_{\mu\in{\cal N}}\sum_{k=0}^\infty a_k\,\mu[\{k\}]<\infty$. This implies that a uniformly $\psi_1$-integrating set ${\cal N}$ is mean bounded in the sense that $\sup_{\mu\in{\cal N}} m_\mu<\infty$. On the other hand an arbitrary set ${\cal N}$ that is ``$p$th moment bounded'' for some $p>1$ is uniformly $\psi_1$-integrating. In particular, a set ${\cal N}$ is uniformly $\psi_1$-integrating if its elements are supported by a common finite set.
{\hspace*{\fill}$\Diamond$\par\bigskip}
\end{enumerate}
\end{remarknorm}

We may now formulate  our main result.
\begin{theorem}\label{asymptotic robustness of lotka}
Let ${\cal N}\subset{\cal N}_1^1$ be such that $m_\mu>1$ for all $\mu\in{\cal N}$. Then the following assertions hold:
\begin{itemize}
    \item[(i)] The sequence $(\widehat m_n)$ is robust on ${\cal N}$ if ${\cal N}$ is locally uniformly $\psi_1$-integrating.

    \item[(ii)] The sequence $(\widehat m_n)$ is uniformly robust on ${\cal N}$ if ${\cal N}$ is uniformly $\psi_1$-integrating and $\inf_{\mu\in{\cal N}}m_\mu>1$.

    \item[(iii)] The sequence $(\widehat m_n)$ is not robust on ${\cal N}$ if the mapping ${\cal N}\ni\mu\mapsto m_\mu$ is not $(d_{\scriptsize{\rm TV}},|\cdot|)$-continuous on all of ${\cal N}$.
\end{itemize}
\end{theorem}

An outline of the proof is given at the end of this section. The detailed arguments are presented in Sections~\ref{Sec Auxiliary results}--\ref{proof of main results}.

\begin{remarknorm}
We note that the statement of the theorem remains the same if we consider a Galton-Watson branching process $(Z_n)$ that is started with $z_0 \in \N$ individuals instead of started with $1$ individual. The modifications that are needed in the proofs in order to show this slightly more general statement are outlined in Section~\ref{Extension to general initial states}.
{\hspace*{\fill}$\Diamond$\par\bigskip}
\end{remarknorm}

In what follows we give a number of illustrative examples.

\begin{examplenorm}\label{binary branching}
Let us consider the set ${\cal N}_{\scriptsize{\rm bin}}$ of all probability measures that are supported by the set $\{0,2\}$. Note that each element $\mu$ of ${\cal N}_{\scriptsize{\rm bin}}$ corresponds to a Galton--Watson process with binary branching. The set ${\cal N}_{\scriptsize{\rm bin}}$ is obviously uniformly $\psi_1$-integrating, such that by part (ii) of Theorem \ref{asymptotic robustness of lotka} the sequence $(\widehat m_n)$ of Lotka--Nagaev estimators is uniformly robust on ${\cal N}_{\scriptsize{\rm bin}}$.

Note that an element $\mu$ of ${\cal N}_{\scriptsize{\rm bin}}$ is uniquely determined by the probability $p:=\mu[\{2\}]$ for $2$ offspring. Also note that the total variation distance of two elements $\mu_1$ and $\mu_2$ of ${\cal N}_{\scriptsize{\rm bin}}$ equals the distance of $p_1:=\mu_1[\{2\}]$ and $p_2:=\mu_2[\{2\}]$, i.e.\ $d_{\scriptsize{\rm TV}}(\mu_1,\mu_2)=|p_1-p_2|$. Thus uniform robustness of the sequence $(\widehat m_n)$ on ${\cal N}_{\scriptsize{\rm bin}}$ means that for every $\eps>0$ there is some $\delta>0$ such that for arbitrary $n \in \N$ and $p_1,p_2 \in [0,1]$ with $|p_1-p_2|\le\delta$ the distributions of the Lotka--Nagaev estimator $\widehat m_n$ under the parameters $p_1$ and $p_2$ are within a Prohorov-distance of $\eps$ of one another. Of course, the same holds true for the distributions of the plug-in estimators $\widehat{p}^{(n)} = \widehat m_n/2$.

For applications this becomes relevant if we want to estimate the true parameter $p_1$ in the ${\cal N}_{\scriptsize{\rm bin}}$ model, but are only able to take observations from a slightly perturbed model with parameter $p_2 \approx p_1$. The above result then tells us that our estimator has ``essentially the same'' distributional properties as it would have with observations from the true model.
{\hspace*{\fill}$\Diamond$\par\bigskip}
\end{examplenorm}

\begin{examplenorm}
Suppose that we would like to estimate $p$ in the model ${\cal N}_{\scriptsize{\rm bin}}$ of the previous example, but in reality the offspring distribution lies in a larger class ${\cal N} \supset {\cal N}_{\scriptsize{\rm bin}}$, i.e.\ our model is misspecified. As a simple example suppose that ${\cal N}$ is the set of all probability measures with support $\{0,2,3\}$. Then ${\cal N}$ is of course still uniformly $\psi_1$-integrating. Note that the total variation distance between an element $\mu_1 \in {\cal N}$ with mass $q>0$ at $3$ and an element $\mu_2 \in {\cal N}_{\scriptsize{\rm bin}}$ that distributes this additional mass among $0$ and $2$ is exactly $q$.

The uniform robustness property obtained by Theorem~\ref{asymptotic robustness of lotka}(ii) tells us then essentially that for $q$ small, i.e.\ if the model is only slightly misspecified, the distribution of $\widehat{m}_n$ (and hence of $\widehat{p}^{(n)}$) is still close to the distribution we would have obtained if our model assumption had been correct.
{\hspace*{\fill}$\Diamond$\par\bigskip}
\end{examplenorm}

\begin{examplenorm}\label{Poisson branching}
The class ${\cal N}_{\text{pois}}$ of Poisson distributions $\Pi_\lambda$, $\lambda>0$, is locally uniformly $\psi_1$-integrating by Remark~\ref{rem ui char}(i). Indeed, if $(\Pi_{\lambda_n})$ is a sequence in ${\cal N}_{\text{pois}}$ such that $\Pi_{\lambda_n} \to \Pi_{\lambda}$ weakly for some $\lambda>0$, we have in particular that $\lambda_n = -\log(\Pi_{\lambda_n}[\{0\}]) \to -\log(\Pi_{\lambda}[\{0\}]) = \lambda$, i.e.\ convergence of the means. By Theorem 2.20 in \cite{VanderVaart1998} this implies that $(\Pi_{\lambda_n})$ is uniformly $\psi_1$-integrating. (Note that in the definition of asymptotic uniform integrability on page 17 in \cite{VanderVaart1998} ``$\limsup$'' can be replaced by ``$\sup$''.)

Again we can argue along similar lines as in Example~\ref{binary branching}. If we want to estimate some true $\lambda_1$, but can observe only from a perturbed model with parameter $\lambda_2 \approx \lambda_1$, the robustness still tells us that the distribution of the estimator $\widehat{\lambda}^{(n)} = \widehat{m}_n$ changes only slightly. However, the influence of the perturbation on this change may now vitally depend on $\lambda_1$ because the robustness is not uniform.
{\hspace*{\fill}$\Diamond$\par\bigskip}
\end{examplenorm}

\begin{examplenorm}\label{polynomial branching}
Consider the class ${\cal N}_{\text{poly}}$ of polynomial distributions $P_p$ with existing expectations, i.e.\ $P_p[\{k\}] = c_p (k+1)^{-p}$, where $p>2$ and $c_p$ is a normalizing constant. If $(P_{p_n})$ is a sequence in ${\cal N}_{\text{poly}}$ such that $P_{p_n} \to P_{p}$ weakly for some $p>2$, we have by $P_{p_n}[\{k\}] \to P_p[\{k\}]$ for $k = 0,1$ that $c_{p_n} \to c_p$ and $p_n \to p$ as $n \to \infty$. Writing $p_{*} = \inf_n p_n > 2$ and $p^{*} = \sup_n p_n < \infty$, we obtain
\begin{equation*}
  \sup_n \sum_{k=\ell}^{\infty} k c_{p_n} (k+1)^{-p_n} \leq \sum_{k=\ell}^{\infty} k c_{p^*} (k+1)^{-p_*} \to 0
  \quad \text{ \ as $\ell \to \infty$.}
\end{equation*}
Thus, again by Remark~\ref{rem ui char}(i), we see that ${\cal N}_{\text{poly}}$ is locally uniformly $\psi_1$-integrating.
{\hspace*{\fill}$\Diamond$\par\bigskip}
\end{examplenorm}

As a corollary of Theorem~\ref{asymptotic robustness of lotka} we may show that $(\widehat m_n)$ is robust on ${\cal N}$ if \emph{and only if} ${\cal N}$ is locally uniformly $\psi_1$-integrating. Recall that the $\psi_1$-weak topology on ${\cal N}_1^1$ is defined to be the coarsest topology for which all mappings $\mu\mapsto\int f\,d\mu$, $f\in\F^1$, are continuous, where $\F^1$ is the set of all maps $f:\N_0\to\R$ with $|f(k)|\le C_f(1+|k|)= C_f \psi_1(k)$ for all $k\in\N_0$ and some finite constant $C_f>0$; see, for instance, Section A.5 in \cite{FoellmerSchied2011}. Of course, the $\psi_1$-weak topology is finer than the weak topology. On the other hand, it was shown (in a more general setting) in Section 3.1 in \cite{Zaehle2014b} that locally uniformly $\psi_1$-integrating sets are exactly those subsets of ${\cal N}_1^1$ on which the relative weak topology and the relative $\psi_1$-weak topology coincide. 

\begin{corollary}\label{asymptotic robustness of lotka - corollary}
Let ${\cal N}\subset{\cal N}_1^1$ be such that $m_\mu>1$ for all $\mu\in{\cal N}.$ Then the sequence $(\widehat m_n)$ is robust on ${\cal N}$ if and only if ${\cal N}$ is locally uniformly $\psi_1$-integrating.
\end{corollary}

\begin{proof}
By part (i) of Theorem \ref{asymptotic robustness of lotka} we know that the sequence $(\widehat m_n)$ is robust on ${\cal N}$ if ${\cal N}$ is locally uniformly $\psi_1$-integrating.

Now assume that the sequence $(\widehat m_n)$ is robust on ${\cal N}$. By part (iii) of Theorem \ref{asymptotic robustness of lotka} it follows that the mapping ${\cal N}\ni\mu\mapsto m_\mu$ is  $(d_{\scriptsize{\rm TV}},|\cdot|)$-continuous and thus continuous with respect to the weak topology on
${\cal N}.$   Suppose that ${\cal N}$ is not locally uniformly $\psi_1$-integrating. This implies that the relative $\psi_1$-weak topology on ${\cal N}$ is (strictly) finer than the relative weak topology on ${\cal N}$. In particular, we can find some $\mu,\mu_1,\mu_2,\ldots\in{\cal N}$ such that $\mu_n\to\mu$ weakly but $\mu_n\not\to\mu$ $\psi_1$-weakly. It is easily seen that $\mu_n\to\mu$ $\psi_1$-weakly if and only if $\mu_n\to\mu$ weakly and $m_{\mu_n}\to m_\mu$. So we obtain $m_{\mu_n}\not\to m_\mu$. This contradicts the weak continuity of $\mu\mapsto m_\mu$ on ${\cal N}$.
\end{proof}

We finish this section by giving an outline of the proof of Theorem~\ref{asymptotic robustness of lotka}(i). The proof strategy for part (ii) is exactly the same and the proof of part (iii) is based on a simple contradiction argument; see Theorem~\ref{hampel-huber converse}.

As mentioned in the introduction robustness of a sequence $(\widehat{m}_n)$ on ${\cal N}$ means equicontinuity of the set of maps $\{{\cal N}\rightarrow{\cal M}_1^+$, $\mu\mapsto\pr^\mu\circ\widehat m_n^{-1}: n\in\N\}$. In Section~5 we show this equicontinuity by separately showing continuity (``finite sample robustness'') and asymptotic equicontinuity (``asymptotic robustness'') of these maps.

Finite sample robustness is shown in Theorem~\ref{hampel-huber generalized - finite sample} by a coupling argument using Strassen's theorem and the fact that close offspring distributions generate close distributions of pairs $(Z_{n-1},Z_n)$ of generation sizes for any $n$ (Lemma~\ref{regularity of joint distributions}).

Asymptotic robustness is a somewhat more involved matter. In Lemma~5.2 we first show that it is enough to prove asymptotic robustness if for each $\widehat{m}_n$ we condition on non-extinction up to time $n-1$. The required asymptotic closeness of the conditional distributions of $\widehat{m}_n$ given $Z_{n-1} > 0$, uniformly over $\mu_2$ from a $\delta$-ball of offspring distributions around each $\mu_1 \in {\cal N}$, is then proved by using the locally uniform conditional weak consistency property of $(\widehat{m}_n)$ (Theorem~\ref{uniform conditional weak consistency for Lotka}) and noting that the remaining distance between $m_{\mu_1}$ and $m_{\mu_2}$ is small (Lemma~\ref{continuity of m}).

The detailed arguments can be found in the following sections. We start with a series of general probabilistic lemmas on Galton--Watson processes in Section~\ref{Sec Auxiliary results}. In Section~\ref{Sec Uniform conditional weak consistency} we show [locally] uniform weak consistency of the Lotka--Nagaev estimator on [locally] uniformly $\psi_1$-integrating sets, conditional on non-extinction. After these preparation we carry out the proof of Theorem \ref{asymptotic robustness of lotka} in Section~\ref{proof of main results}. Finally, in Section \ref{Extension to general initial states} we summarize the modifications necessary to see that Theorem  \ref{asymptotic robustness of lotka} also holds for Galton--Watson processes with general initial states.


\section{Auxiliary lemmas about Galton--Watson processes}\label{Sec Auxiliary results}

\begin{lemma}\label{continuity of m}
(i) Let ${\cal N}\subset{\cal N}_1^1$ be a locally uniformly $\psi_1$-integrating set. Then the mapping ${\cal N}\ni\mu\mapsto m_\mu$ is $(d_{\scriptsize{\rm TV}},|\,\cdot\,|)$-continuous.

(ii) If ${\cal N}\subset{\cal N}_1^1$ is even uniformly $\psi_1$-integrating, then the mapping ${\cal N}\ni\mu\mapsto m_\mu$ is uniformly $(d_{\scriptsize{\rm TV}},|\,\cdot\,|)$-continuous.
\end{lemma}

\begin{proof}
We first prove part (i). Fix $\varepsilon>0$ and $\mu_1\in{\cal N}$. Since ${\cal N}$ was assumed to be locally uniformly $\psi_1$-integrating, we can find some $\delta>0$ and $\ell_\varepsilon\in\N$ such that for every $\mu_2\in{\cal N}$ with $d_{\scriptsize{\rm TV}}(\mu_1,\mu_2)\le\delta$, we have
$\sum_{k=\ell_\varepsilon}^{\infty} k\,\mu_2[\{k\}] < \varepsilon/4.$ It follows that
\begin{eqnarray*}
    |m_{\mu_1}-m_{\mu_2}|
    & \le & \sum_{k=1}^\infty k\,\big|\mu_1[\{k\}]-\mu_2[\{k\}]\big|\\
    & \le & \ell_\varepsilon\sum_{k=1}^{\ell_\varepsilon} \big|\mu_1[\{k\}]-\mu_2[\{k\}]\big|\,+\,\sum_{k=\ell_\varepsilon+1}^{\infty}k\,\big|\mu_1[\{k\}]-\mu_2[\{k\}]\big|\\
    & \le & \ell_\varepsilon\,2\,d_{\scriptsize{\rm TV}}(\mu_1,\mu_2)\,+\,\varepsilon/2.
\end{eqnarray*}
Thus, choosing $\delta_\varepsilon:=\min\{\delta;\ell_\varepsilon^{-1}\varepsilon/4\}$ we have that $d_{\scriptsize{\rm TV}}(\mu_1,\mu_2)\le\delta_\varepsilon$ implies $|m_{\mu_1}-m_{\mu_2}|\le\varepsilon.$  This completes the proof of part (i).

Part (ii) can be shown analogously. Set (informally) $\delta:=\infty$ and note that $\ell_\varepsilon$ can be chosen independently of $\mu_1$ when ${\cal N}$ is uniformly $\psi_1$-integrating.
\end{proof}


Let us fix some more notation regarding the Galton--Watson process. We let
$$
    f_{\mu}(s)\,:=\,\sum_{k \in \N_0} s^k \mu[\{k\}], \quad 0 \leq s \leq 1,
$$
be the generating function of the offspring distribution $\mu$. We also use $f_{\mu}^{(n)}$ to denote the $n$th iterate of $f_{\mu}$, which is the generating function of $Z_n$ (recall that $Z_0=1$). By $q_{\mu}$ we denote the extinction probability of the associated Galton--Watson branching process, that is,
$$
    q_{\mu}\,:=\,\pr^\mu[Z_n=0\mbox{ for some }n\in\N].
$$
Except for some of the lemmas in the present section, we assume in this article that $m_{\mu} > 1$. Recall that $q_\mu$ is then the unique solution of $f_{\mu}(s) = s$ in $s \in [0,1)$. The generating function $f_{\mu}$ is strictly increasing and strictly convex, which implies $f_{\mu}'(q_\mu) < 1$. Furthermore we have $f_{\mu}^{(n)}(s) \nearrow q_\mu$ as $n \to \infty$ for every $s \in [0,q)$. See \cite{AthreyaNey1972}, Section~I.3 and~I.5, for this and similar basic results.

\begin{lemma}\label{uniformextinction}
(i) Let ${\cal N}\subset{\cal N}_1^1$ be a locally uniformly $\psi_1$-integrating set with $m_\mu>1$ for all $\mu\in{\cal N}$. Then for every $\mu_1\in{\cal N}$ there exist some $p>0$ and $\delta>0$ such that for all $\mu_2\in{\cal N}$ with $d_{\scriptsize{\rm TV}}(\mu_1,\mu_2)\le\delta$,
\begin{eqnarray}
\label{unifboundextinction i}
    \quad q_{\mu_2} &\leq& 1-p,\\
\label{unifbound-deriv-at-extinction i}
    f_{\mu_2}'(q_{\mu_2})&\leq& 1-p.
\end{eqnarray}

(ii) If ${\cal N}\subset{\cal N}_1^1$ is even uniformly $\psi_1$-integrating with $\inf_{\mu\in{\cal N}}m_\mu>1$, then there exists a $p>0$ such that
\begin{eqnarray}
\label{unifboundextinction ii}
\sup_{\mu\in{\cal N}} q_{\mu} &\leq& 1-p,\\
\label{unifbound-deriv-at-extinction ii}
\sup_{\mu\in{\cal N}} f_{\mu}'(q_{\mu})&\leq& 1-p.
\end{eqnarray}

\end{lemma}

\begin{proof}
We first prove part (i). Let $\mu_1\in{\cal N}$. We start by showing a locally uniform continuity of $f_\mu'$ at $1$ and $\mu_1$, meaning that for all $\varepsilon>0$ there exist some $\delta_1>0$ and $\delta_2>0$ such that for all $\mu_2\in {\cal N}$ with $d_{\scriptsize{\rm TV}}(\mu_1,\mu_2)\le\delta_1$,
\begin{equation}\label{f'unifcont i}
    |f_{\mu_2}'(1)-f_{\mu_2}'(s)|\,\leq\,\varepsilon \quad \text{ for all } s  \in [1-\delta_2, 1].
\end{equation}
Indeed, by the assumption on ${\cal N}$ we can choose for fixed $\varepsilon>0$ some $\delta_1=\delta_1(\varepsilon)>0$ and $\ell=\ell(\varepsilon)\in\N$ such that $\sum_{k=\ell+1}^\infty k\,\mu_2[\{k\}]\le \varepsilon/4$ for all $\mu_2\in {\cal N}$ with $d_{\scriptsize{\rm TV}}(\mu_1,\mu_2)\le\delta_1$. Set $\delta_2=\delta_2(\varepsilon):=\frac{\varepsilon}{2\ell^2}$. Then we have for all $\mu_2\in {\cal N}$ with $d_{\scriptsize{\rm TV}}(\mu_1,\mu_2)\le\delta_1$ and all $s \in [1-\delta_2, 1]$,
\begin{eqnarray*}
    |f_{\mu_2}'(1) - f_{\mu_2}'(s)|
    & = & \Big|\sum_{k=1}^\infty k (1-s^{k-1}) \mu_2[\{k\}]\Big|\\
    & \le & \sum_{k=1}^\ell k (k-1)\,(1-s)\,\mu_2[\{k\}]\, + \,2 \sum_{k=\ell+1}^\infty k\,\mu_2[\{k\}]\\
    & \le & \ell^2 \delta_2\,+\,2\,\frac{\varepsilon}{4}\\[1mm]
    & = & \varepsilon, 
\end{eqnarray*}
where we have used that $1-s^{k-1}\leq (k-1)(1-s)$ for $s \in [0,1]$. This shows (\ref{f'unifcont i}).

Next, recall that $m_{\mu_1}>1$ and choose $\varepsilon>0$ small enough such that $2\varepsilon<m_{\mu_1}-1$. By Lemma~\ref{continuity of m} we can find some $\delta_3=\delta_3(\varepsilon)>0$ such that for all $\mu_2\in{\cal N}$ with $d_{\scriptsize{\rm TV}}(\mu_1,\mu_2)\le\delta_3$,
\begin{equation}\label{diff m mu1 m mu2}
    |m_{\mu_1}-m_{\mu_2}|\,\le\,\varepsilon.
\end{equation}
Now we use (\ref{f'unifcont i}) and (\ref{diff m mu1 m mu2}) in order to obtain some $\delta_1\in(0,\delta_3]$ and $\delta_2\in(0,\delta_3]$ such that for all $\mu_2\in{\cal N}$ with $d_{\scriptsize{\rm TV}}(\mu_1,\mu_2)\le\delta_1$ and all $s \in [1-\delta_2, 1]$,
\begin{eqnarray}
    f_{\mu_2}'(s)
     & = & f_{\mu_2}'(1)-(f_{\mu_2}'(1) - f_{\mu_2}'(s))\nonumber\\
     & = & m_{\mu_2}-(f_{\mu_2}'(1) - f_{\mu_2}'(s))\nonumber\\
     & \ge & m_{\mu_1}-2\varepsilon~>~1. \label{lower bound for f'}
\end{eqnarray}
From this we get in particular that for all $\mu_2\in{\cal N}$ with $d_{\scriptsize{\rm TV}}(\mu_1,\mu_2)\le\delta_1$,
\begin{equation*}
    f_{\mu_2}(s) \,\leq\, 1- (1-s) (m_{\mu_1}-2\varepsilon)\,<\,s  \quad \text{ for all } s  \in [1-\delta_2, 1].
\end{equation*}
Since $q_{\mu_2}<1$ and $f_{\mu_2}(q_{\mu_2})= q_{\mu_2}$, this implies that $q_{\mu_2}< 1-\delta_2$ for all $\mu_2\in{\cal N}$ with $d_{\scriptsize{\rm TV}}(\mu_1,\mu_2)\le\delta_1$, which shows (\ref{unifboundextinction i}) with $p:=\delta_2$ and $\delta:=\delta_1$. Also, using the convexity of $f_{\mu_2}$ and the fact that $f_{\mu_2}(1-\delta_2) \leq 1-\delta_2 (m_{\mu_1}-2\varepsilon)$ it is easy to see that for all $\mu_2\in{\cal N}$ with $d_{\scriptsize{\rm TV}}(\mu_1,\mu_2)\le\delta_1$
\begin{equation*}
    f_{\mu_2}'(q_{\mu_2})\,\le\,  \frac{ f_{\mu_2}(1-\delta_2)-f_{\mu_2}(0) }{1-\delta_2}
    \,\le\,\frac{1-\delta_2 (m_{\mu_1}-2\varepsilon)}{1-\delta_2}\,<\,1,
\end{equation*}
where we have bounded the left hand side by the slope of the line connecting $(0,0)$ with $(1-\delta_2, 1-\delta_2 (m_{\mu_1}-2\varepsilon))$. This shows (\ref{unifbound-deriv-at-extinction i}) with $p:=1-(1-\delta_2 (m_{\mu_1}-2\varepsilon))/(1-\delta_2)$ and $\delta:=\delta_1$, and completes the proof of part (i).

Part (ii) can be shown analogously. Set (informally) $\delta_1:=\delta_3:=\infty$, skip \eqref{diff m mu1 m mu2}, and replace $m_{\mu_1}$ by $\underline{m}:=\inf_{\mu \in {\cal N}} m_{\mu}>1$ in what follows.
\end{proof}


\begin{lemma}\label{unifgrowth}
(i) Let ${\cal N}\subset{\cal N}_1^1$ be a locally uniformly $\psi_1$-integrating set with $m_\mu>1$ for all $\mu\in{\cal N}$. Then for every $\mu_1\in{\cal N}$, $k\in\N$, $\varepsilon>0$ there exist some $\delta>0$ and $n_0 \in \N$ such that 
\begin{equation}\label{uniform conditional weak consistency for Lotka - proof - 50 - prime}
    \mu_2\in{\cal N},\quad d_{\scriptsize{\rm TV}}(\mu_1,\mu_2)\le\delta\quad\Longrightarrow\quad\pr^{\mu_2}[Z_{n}=k|Z_{n}>0]\,\le\,\varepsilon\quad\mbox{ for all }n\ge n_0.
\end{equation}

(ii) If ${\cal N}\subset{\cal N}_1^1$ is even uniformly $\psi_1$-integrating with $\inf_{\mu\in{\cal N}}m_\mu>1$, then for every $k\in\N$ and $\varepsilon>0$ there exists some $n_0 \in \N$ such that
\begin{equation}\label{uniform conditional weak consistency for Lotka - proof - 50 - pprime}
    \sup_{\mu\in{\cal N}}\,\pr^\mu[Z_{n}=k|Z_{n}>0]\,\le\,\varepsilon\quad\mbox{ for all }n\ge n_0.
\end{equation}
\end{lemma}

\begin{proof}
We first prove part (i). Fix $\mu_1\in{\cal N}$, $k\in\N$, and $\varepsilon>0$. Let $p\in(0,1)$ and $\delta>0$ be as in part (i) of Lemma \ref{uniformextinction}, and $\mu_2\in{\cal N}$ with $d_{\scriptsize{\rm TV}}(\mu_1,\mu_2)\le\delta$. Let $A$ be the event that a Galton--Watson branching process survives and let $B$ be the event that it goes extinct. We have by $\pr^\mu[Z_n>0]\ge 1-q_\mu$ that
\begin{eqnarray}
\label{from alive at n to overall survival}
    \pr^{\mu_2}[ Z_n =k | Z_n>0]
    & = & \pr^{\mu_2}[ \{Z_n =k\}\cap A | Z_n>0]\,+\,\pr^{\mu_2}[\{Z_n =k\}\cap B | Z_n>0] \nonumber\\
    & \leq & \frac{\pr^{\mu_2}[ \{Z_n =k\}\cap A\cap\{Z_n>0\}]}{\pr^{\mu_2}[  Z_n>0]}\,+\,\pr^{\mu_2}[ B | Z_n>0] \nonumber\\
    & \leq & \frac{\pr^{\mu_2}[ \{Z_n =k\}\cap A]}{1-q_{\mu_2}} + \pr^{\mu_2}[ B | Z_n>0] \nonumber\\
    & = & \pr^{\mu_2}[ Z_n =k| A] + \pr^{\mu_2}[ B | Z_n>0].
\end{eqnarray}
For bounding the first term we decompose $Z_n=Z_n^{(1)}  + Z_n^{(2)}$ where $Z_n^{(1)}$ is the number of particles among $Z_n$ with infinite line of descent. We then use the fact that $Z_n^{(1)}$ under $ \pr^{\mu_2}[\,\cdot\,|A]$ has the same distribution as $Z_n$ under $\pr^{\widehat{\mu}_2}$ where $\widehat{\mu}_2$ is an offspring distribution with generating function
\begin{equation}\label{f_hat{mu}}
    f_{\widehat{\mu}_2}(s)\,=\,\frac{f_{\mu_2}((1-q_{\mu_2})s + q_{\mu_2})- q_{\mu_2}}{1-q_{\mu_2}}\,,
\end{equation}
see Theorem I.12.1 of \cite{AthreyaNey1972}. Note that $f_{\widehat{\mu}_2}$ results from taking $f_{\mu_2}$ on the square $[q_{\mu_2},1]^2$ and stretching it linearly to the unit square $[0,1]^2.$ Naturally, we have that the corresponding Galton--Watson branching process is supercritical with $\widehat{\mu}_2[\{0\}]=f_{\widehat{\mu}_2}(0)=0$ and so also $q_{\widehat{\mu}_2}=0$. By (\ref{unifbound-deriv-at-extinction i}) of Lemma \ref{uniformextinction} and the choice of $p$, 
\begin{equation}\label{hatmuunif}
    \widehat{\mu}_2[\{1\}]\,=\,f_{\widehat{\mu}_2}'(0)\,=\,f_{\mu_2}'(q_{\mu_2})\,\leq\,1-p.
\end{equation}
Under $\pr^{\widehat{\mu}_2}$, the process $Z_n$ is a.s.\ increasing in $n$. The probability that it increases by a positive quantity is at least $1-f_{\mu_2}'(q_{\mu_2})\geq p$. Thus, if ${\rm B}_{n,p}$ denotes the binomial distribution with parameters $n$ and $p$ we have
\begin{eqnarray}\label{binomial bound}
    \pr^{\mu_2}[ Z_n =k| A]
    & \leq & \pr^{\mu_2}[ Z_n \leq k | A]  \nonumber\\
    &\leq & \pr^{\mu_2}[ Z_n^{(1)} \leq k | A]  \nonumber\\
    &=& \pr^{\widehat{\mu}_2}[ Z_n \leq k ] \nonumber \\
    & \leq & {\rm B}_{n,p}[\{0,\ldots,k\}] \nonumber\\
    & \le & \varepsilon/2
\end{eqnarray}
for all $n\ge n_1$ for some sufficiently large $n_1\in\N$.

It remains to bound the probability of extinction given that $Z_n>0$. Here, we rewrite
\begin{eqnarray*}
    \pr^{\mu_2}[ B | Z_n>0]\,=\,\frac{\pr^{\mu_2}[ Z_n>0| B] \cdot \pr^{\mu_2}[B]}{\pr^{\mu_2}[ Z_n>0]}\,\leq\,\pr^{\mu_2}[ Z_n>0| B] \, \frac{q_{\mu_2}}{1-q_{\mu_2}}\,.
\end{eqnarray*}
Due to (\ref{unifboundextinction i}) of Lemma  \ref{uniformextinction} it then remains to bound $\pr^{\mu_2}[ Z_n>0| B]$ uniformly. Here, we use the fact that $Z_n$ is under $\pr^{\mu_2}[\,\cdot\,| B]$ a subcritical Galton--Watson branching process with offspring distribution $\mu_2^*$ described via its generating function
\begin{equation*}
    f_{\mu_2^*}(s) = \frac{1}{q_{\mu_2}} f_{\mu_2}(s q_{\mu_2}),
\end{equation*}
see Theorem I.12.3 of \cite{AthreyaNey1972}. Therefore, we have $m_{\mu_2^*} = f_{\mu_2^*}'(1) = f_{\mu_2}'(q_{\mu_2}) \leq 1-p$ 
by (\ref{unifbound-deriv-at-extinction i}) of Lemma \ref{uniformextinction} and the choice of $p$. Thus, by Markov's inequality
\begin{equation*}
    \pr^{\mu_2}[ Z_n>0| B]\,=\,\pr^{\mu_2^*}[ Z_n>0] \,=\,\pr^{\mu_2^*}[ Z_n\geq 1] \,\leq\,  \ex^{\mu_2^*}[Z_n] \, =\, m_{\mu_2^*}^n \,\leq\, (1-p)^n\,\le\,\varepsilon/2
\end{equation*}
for all $n\ge n_0$ for some sufficiently large $n_0\ge n_1$. This completes the proof of part (i).

Part (ii) can be shown analogously, using (\ref{unifboundextinction ii})--(\ref{unifbound-deriv-at-extinction ii}) instead of (\ref{unifboundextinction i})--(\ref{unifbound-deriv-at-extinction i}).
\end{proof}


\begin{lemma}\label{proof of asymptotic robustness of lotka - lemma 1}
(i) For every $\mu_1\in{\cal N}_1^1$ with $m_{\mu_1} > 1$ and every $\varepsilon>0$ we can find a $\delta>0$ such that
\begin{eqnarray}
    \lefteqn{\mu_2\in{\cal N}_1^1, \quad d_{\mbox{\scriptsize{\rm TV}}}(\mu_1,\mu_2)\le\delta}\nonumber\\
    & & \Longrightarrow\quad |\pr^{\mu_1}[Z_{n}=0]-\pr^{\mu_2}[Z_{n}=0]|\,\le\,\varepsilon \quad\mbox{for all }n \in \N.\label{proof of asymptotic robustness of lotka - 30}
\end{eqnarray}

(ii) Let ${\cal N}\subset{\cal N}_1^1$ be a uniformly $\psi_1$-integrating set with $\inf_{\mu\in{\cal N}}m_\mu>1$. Then for every $\varepsilon>0$ we can find a $\delta>0$ such that
\begin{eqnarray}
    \lefteqn{\mu_1,\mu_2\in{\cal N}, \quad d_{\mbox{\scriptsize{\rm TV}}}(\mu_1,\mu_2)\le\delta}\nonumber\\
    & & \Longrightarrow\quad |\pr^{\mu_1}[Z_{n}=0]-\pr^{\mu_2}[Z_{n}=0]|\,\le\,\varepsilon \quad\mbox{for all }n \in \N.\label{proof of asymptotic robustness of lotka - 35}
\end{eqnarray}
\end{lemma}

\begin{proof}
First note that for any $\mu_1,\mu_2 \in \mathcal{N}_1^1$, we have
\begin{eqnarray}
    \lefteqn{|f_{\mu_1}(s) - f_{\mu_2}(s)|}\nonumber\\
    & = & \Bigl| \sum_{k\in\N_0} s^k \bigl( \mu_1[\{k\}] - \mu_2[\{k\}] \bigr) \Bigr|\nonumber\\
    & \leq & \max \biggl\{ \sum_{\substack{k\in\N_0\\ \mu_1(k) > \mu_2(k)}} s^k \bigl( \mu_1(k) - \mu_2(k) \bigr), \sum_{\substack{k\in\N_0\\ \mu_1(k) < \mu_2(k)}} s^k \bigl( \mu_2(k) - \mu_1(k) \bigr) \biggr\} \nonumber\\
    & \leq & d_{\scriptsize{\rm TV}}(\mu_1,\mu_2)\label{gf induction 1}
\end{eqnarray}
by the fact that
\begin{equation*}
  \sum_{\substack{k\in\N_0\\ \mu_1(k) > \mu_2(k)}} \bigl( \mu_1(k) - \mu_2(k) \bigr)
  \,=\sum_{\substack{k\in\N_0\\ \mu_1(k) < \mu_2(k)}} \bigl( \mu_2(k) - \mu_1(k) \bigr)
  \,=\,d_{\scriptsize{\rm TV}}(\mu_1,\mu_2).
\end{equation*}

We now show part (i). Let $\varepsilon > 0$ and $\mu_1 \in \mathcal{N}_1^1$ with $m_{\mu_1} > 1$.
Since $f_{\mu_1}'(q_{\mu_1}) < 1$ and $f'_{\mu_1}$ is continuous, we may choose $\bar{q} > q_{\mu_1}$ such that $\bar{\gamma} := (f_{\mu_1})'(\bar{q}) < 1$. Set
\begin{equation}\label{def delta}
    \delta\, := \, \min \bigl\{(\bar{q} - f_{\mu_1}(\bar{q}))/2\,;\, (1-\bar{\gamma}) \varepsilon\bigr\} >\,0.
\end{equation}
Letting $\mu_2 \in \mathcal{N}_1^1$ with $d_{\mbox{\scriptsize{\rm TV}}}(\mu_1,\mu_2) \leq \delta$, we obtain by (\ref{gf induction 1}), (\ref{def delta}) and $f_{\mu_1}(\bar{q})<\bar{q}$ that
$$
    f_{\mu_2}(\bar{q})\,\leq \,f_{\mu_1}(\bar{q}) + |f_{\mu_2}(\bar{q})-f_{\mu_1}(\bar{q})| \,\le\,f_{\mu_1}(\bar{q}) + \frac{\bar{q} - f_{\mu_1}(\bar{q})}{2} \,<\, \bar{q}.
$$
Since $f_{\mu_2}(s)<s$ holds if and only if $s>q_{\mu_2}$, we conclude $q_{\mu_2}< \bar{q}$. Note that $0 \leq f^{(1)}_{\mu_i}(0) \leq f^{(2)}_{\mu_i}(0) \leq \cdots \leq q_{\mu_i} \le \bar{q}$, $i=1,2$. Furthermore, since $f_{\mu_1}$ is convex, it is Lipschitz continuous on $[0,\bar{q}]$ with constant $\bar{\gamma} < 1$.
Therefore we have for $n \geq 2$
\begin{eqnarray}
  \lefteqn{|f^{(n)}_{\mu_1}(0) - f^{(n)}_{\mu_2}(0)|}\nonumber\\
   & \leq & \bigl| f_{\mu_1}(f^{(n-1)}_{\mu_1}(0)) - f_{\mu_1}(f^{(n-1)}_{\mu_2}(0)) \bigr|
         \,+\, \bigl| f_{\mu_1}(f^{(n-1)}_{\mu_2}(0)) - f_{\mu_2}(f^{(n-1)}_{\mu_2}(0)) \bigr|\nonumber\\
   & \leq & \bar{\gamma}\, \bigl| f^{(n-1)}_{\mu_1}(0) - f^{(n-1)}_{\mu_2}(0) \bigr|\, +\, d_{\scriptsize{\rm TV}}(\mu_1,\mu_2).\label{gf induction 2}
\end{eqnarray}
For the case $n=1$ we obtain by \eqref{gf induction 1} that
\begin{equation*}
|f^{(1)}_{\mu_1}(0) - f^{(1)}_{\mu_2}(0)|= |f_{\mu_1}(0) - f_{\mu_2}(0)| \leq d_{\scriptsize{\rm TV}}(\mu_1,\mu_2).
\end{equation*}
By induction we obtain from this and inequality~\eqref{gf induction 2} that
\begin{equation}
\label{gf result}
    |\pr^{\mu_1}[Z_{n}=0]-\pr^{\mu_2}[Z_{n}=0]|\,=\,|f^{(n)}_{\mu_1}(0) - f^{(n)}_{\mu_2}(0)|\,
   \leq \, \Big( \sum_{k=0}^n \bar{\gamma}^k \Big) d_{\scriptsize{\rm TV}}(\mu_1,\mu_2) \, \leq \, \varepsilon
\end{equation}
for all $n \in \N$. This completes the proof of part (i).

Part (ii) can be shown in a similar way. Set $\delta := (1-\gamma^{*})\eps$, where $\gamma^{*} := \sup_{\mu \in {\cal N}} f'_{\mu}(q_{\mu}) < 1$ by Lemma~\ref{uniformextinction}(ii). Let $\mu_1,\mu_2 \in \mathcal{N}$ with $d_{\mbox{\scriptsize{\rm TV}}}(\mu_1,\mu_2) \leq \delta$ and set $q^{*} := \max(q_{\mu_1},q_{\mu_2})$. By convexity the function $f_{\mu_i}$ is Lipschitz continuous on $[0,q^{*}]$ with constant $f'_{\mu_i}(q^{*})$ for $i = 1,2$. Hence using that $f^{(n)}_{\mu_i}(0) \leq q_{\mu_i} \leq q^*$  inequality~\eqref{gf induction 2} can be replaced by
\begin{equation}
  |f^{(n)}_{\mu_1}(0) - f^{(n)}_{\mu_2}(0)|
   \leq f'_{\mu_i}(q^{*}) \, \bigl| f^{(n-1)}_{\mu_1}(0) - f^{(n-1)}_{\mu_2}(0) \bigr|\, +\, d_{\scriptsize{\rm TV}}(\mu_1,\mu_2) \quad \text{for $i = 1,2$}.
\end{equation}
Since $\min_{i \in \{1,2\}} f'_{\mu_i}(q^{*}) \leq \gamma^{*}$, we obtain that inequality \eqref{gf result} holds for all $n \in \N$ with $\bar{\gamma}$ replaced by $\gamma^{*}$.
\end{proof}


Now, let ${\cal N}_1^{1,n}$ be the set of all probability measures on $\N_0^n$ with marginal distributions in ${\cal N}_1^1$ and $d_{\scriptsize{\rm TV}}^{(n)}$ the total variation distance on ${\cal N}_1^{1,n}$. The following lemma shows in particular that the mapping ${\cal N}_1^1\to{\cal N}_1^{1,n}$, $\mu\mapsto \pr^\mu\circ(Z_1,\ldots,Z_n)^{-1}$ is $(d_{\scriptsize{\rm TV}},d_{\scriptsize{\rm TV}}^{(n)})$-continuous.

\begin{lemma}\label{regularity of joint distributions}
For every $\mu_1,\mu_2\in{\cal N}_1^1$ and $n\in\N$ we have
\begin{equation}\label{regularity of joint distributions - eq}
    d_{\scriptsize{\rm TV}}^{(n)}\big(\pr^{\mu_1}\circ(Z_1,\ldots,Z_n)^{-1},\pr^{\mu_2}\circ(Z_1,\ldots,Z_n)^{-1}\big)\,\le\,C_n(\mu_1,\mu_2) \,d_{\scriptsize{\rm TV}}(\mu_1,\mu_2),
\end{equation}
where $C_n(\mu_1,\mu_2) := \min\{\sum_{i=1}^nm_{\mu_1}^{i-1}, \sum_{i=1}^nm_{\mu_2}^{i-1}\}.$
\end{lemma}

\begin{proof}
Let $\mu_1,\mu_2\in{\cal N}^1_1$, $n\in\N$, and $(k_1,\ldots,k_n)\in\N_0^n$.
By the Markov property we have
\begin{eqnarray}
    \lefteqn{\pr^{\mu_i}[(Z_1,\ldots,Z_n)=(k_1,\ldots,k_n)]}\nonumber\\[1mm]
    & = & \pr^{\mu_i}[Z_1=k_1]\cdot\pr^{\mu_i}[Z_2=k_2|Z_1=k_1]\cdots\pr^{\mu_i}[Z_n=k_n|Z_{n-1}=k_{n-1}]\nonumber\\
    & = & \prod_{j=1}^n\mu_i^{*k_{j-1}}[\{k_{j}\}]\label{regularity of joint distributions - proof - 1}
\end{eqnarray}
for $i=1,2$, where we set $k_0:=1.$ Here $\mu_i^{*k}$ denotes the $k$th convolution of the measure $\mu_i$ and we set $\mu_i^{*1}:=\mu_i$.
Note furthermore that for $x_j, y_j \geq 0$,
\begin{eqnarray}
   \Big|\prod_{j=1}^nx_j-\prod_{j=1}^ny_j\Big|
   & = & \Big| \sum_{i=1}^n \Big[ \Big( \prod_{j=1}^{i-1} y_j \Bigr) x_i
   \Big( \prod_{\ell=i+1}^{n}x_\ell \Big) - \Big( \prod_{j=1}^{i-1} y_j \Big) y_i
   \Big( \prod_{\ell=i+1}^{n}x_\ell \Big) \Big] \Big|\nonumber\\
   & \le & \sum_{i=1}^n|x_i-y_i|\prod_{j=1}^{i-1}y_j\prod_{\ell=i+1}^{n}x_\ell.\label{regularity of joint distributions - proof - 2}
\end{eqnarray}

Combining \eqref{regularity of joint distributions - proof - 1} and \eqref{regularity of joint distributions - proof - 2} we obtain
\begin{eqnarray}
    \lefteqn{2\,d_{\scriptsize{\rm TV}}^{(n)}\big(\pr^{\mu_1}\circ(Z_1,\ldots,Z_n)^{-1},\pr^{\mu_2}\circ(Z_1,\ldots,Z_n)^{-1}\big)}\nonumber\\[1mm]
    & = & \sum_{(k_1,\ldots,k_n)\in\N_0^n}\big|\pr^{\mu_1}[(Z_1,\ldots,Z_n)=(k_1,\ldots,k_n)]-\pr^{\mu_2}[(Z_1,\ldots,Z_n)=(k_1,\ldots,k_n)]\big|\nonumber\\
    & = & \sum_{(k_1,\ldots,k_n)\in\N_0^n}\Big|\prod_{j=1}^n{\mu_1}^{*k_{j-1}}[\{k_{j}\}]-\prod_{j=1}^n\mu_2^{*k_{j-1}}[\{k_{j}\}]\Big|\nonumber\\
    & \le & \sum_{(k_1,\ldots,k_n)\in\N_0^n}\sum_{i=1}^n\big|\mu_1^{*k_{i-1}}[\{k_{i}\}]-\mu_2^{*k_{i-1}}[\{k_{i}\}]\big| \,\prod_{j=1}^{i-1}\mu_2^{*k_{j-1}}[\{k_{j}\}]\prod_{\ell=i+1}^{n}\mu_1^{*k_{\ell-1}}[\{k_{\ell}\}]\nonumber\\
    & = & \sum_{i=1}^n\Big\{\sum_{k_1\in\N_0}\cdots\sum_{k_n\in\N_0}\big|\mu_1^{*k_{i-1}}[\{k_{i}\}]-\mu_2^{*k_{i-1}}[\{k_{i}\}]\big| \,\prod_{j=1}^{i-1}\mu_2^{*k_{j-1}}[\{k_{j}\}]\prod_{\ell=i+1}^{n}\mu_1^{*k_{\ell-1}}[\{k_{\ell}\}]\Big\}\nonumber\\
    & =: & \sum_{i=1}^n S_i(\mu_1,\mu_2).\label{regularity of joint distributions - proof - 10}
\end{eqnarray}
Using $\sum_{k_{n}\in\N_0} \mu_1^{*k_{n-1}}[\{k_{n}\}]=1$, we have for $2\le i\le n-1$ that
\begin{eqnarray}
    \lefteqn{S_i(\mu_1,\mu_2)}\nonumber\\
    & = & \sum_{k_1\in\N_0}\cdots\sum_{k_{n-1}\in\N_0}\big|\mu_1^{*k_{i-1}}[\{k_{i}\}]-\mu_2^{*k_{i-1}}[\{k_{i}\}]\big|\sum_{k_n\in\N_0}\,\prod_{j=1}^{i-1}\mu_2^{*k_{j-1}}[\{k_{j}\}] \prod_{\ell=i+1}^{n}\mu_1^{*k_{\ell-1}}[\{k_{\ell}\}]\nonumber\\
    & = & \sum_{k_1\in\N_0}\cdots\sum_{k_{n-1}\in\N_0}\big|\mu_1^{*k_{i-1}}[\{k_{i}\}]-\mu_2^{*k_{i-1}}[\{k_{i}\}]\big|\, \prod_{j=1}^{i-1}\mu_2^{*k_{j-1}}[\{k_{j}\}]\prod_{\ell=i+1}^{n-1}\mu_1^{*k_{\ell-1}}[\{k_{\ell}\}]\nonumber\\
    & = & \sum_{k_1\in\N_0}\cdots\sum_{k_{i}\in\N_0}\big|\mu_1^{*k_{i-1}}[\{k_{i}\}]-\mu_2^{*k_{i-1}}[\{k_{i}\}]\big|\,\prod_{j=1}^{i-1}\mu_2^{*k_{j-1}}[\{k_{j}\}],\nonumber
\end{eqnarray}
where the last step follows by iteration of the previous two steps. Since $d_{\scriptsize{\rm TV}}(\mu_1^{*k},\mu_2^{*k})\le k\,d_{\scriptsize{\rm TV}}(\mu_1,\mu_2)$ for every $k$, we can proceed as
\begin{eqnarray}
    & \le & \sum_{k_1\in\N_0}\cdots\sum_{k_{i-1}\in\N_0}2\,k_{i-1}\,d_{\scriptsize{\rm TV}}(\mu_1,\mu_2)\,\prod_{j=1}^{i-1}\mu_2^{*k_{j-1}}[\{k_{j}\}]\nonumber\\
    & \le & 2\,d_{\scriptsize{\rm TV}}(\mu_1,\mu_2)\sum_{k_1\in\N_0}\cdots\sum_{k_{i-1}\in\N_0}k_{i-1}\,\,\mu_2^{*k_{i-2}}[\{k_{i-1}\}]\,\mu_2^{*k_{i-3}}[\{k_{i-2}\}]\cdots\mu_2^{*k_{0}}[\{k_{1}\}]\nonumber\\
    & = & 2\,d_{\scriptsize{\rm TV}}(\mu_1,\mu_2)\sum_{k_1\in\N_0}\cdots\sum_{k_{i-2}\in\N_0}k_{i-2}\,m_{\mu_2}\,\mu_2^{*k_{i-3}}[\{k_{i-2}\}]\cdots\mu_2^{*k_{0}}[\{k_{1}\}]\nonumber\\
    & = & 2\,d_{\scriptsize{\rm TV}}(\mu_1,\mu_2)\,m_{\mu_2}^{i-1},\label{regularity of joint distributions - proof - 20}
\end{eqnarray}
where the last step follows by iteration. Note that this is again true for $2\le i\le n-1$ since the expression $i-3$ only appears in the above in order to illustrate the iteration for larger $i \geq 3.$ Analogously we obtain
\begin{equation}\label{regularity of joint distributions - proof - 30}
    S_1(\mu_1,\mu_2)\,\le\,2\,d_{\scriptsize{\rm TV}}(\mu_1,\mu_2)\quad\mbox{ and }\quad S_n(\mu_1,\mu_2)\,\le\,2\,d_{\scriptsize{\rm TV}}(\mu_1,\mu_2)\,m_{\mu_2}^{n-1}.
\end{equation}
Now, (\ref{regularity of joint distributions - proof - 10})--(\ref{regularity of joint distributions - proof - 30}) imply (\ref{regularity of joint distributions - eq})
with $C_n(\mu_1,\mu_2)$ replaced by  $\sum_{i=1}^nm_{\mu_2}^{i-1}.$ Due to symmetry the proof for (\ref{regularity of joint distributions - eq})
with $C_n(\mu_1,\mu_2)$ replaced by  $\sum_{i=1}^nm_{\mu_1}^{i-1}$ is analogous, which shows  (\ref{regularity of joint distributions - eq}).
\end{proof}


\begin{lemma}\label{Uniform WLLN}
(i) Let ${\cal N}\subset{\cal N}_1^1$ be a locally uniformly $\psi_1$-integrating set. Then for every $\mu_1\in{\cal N}$, $\varepsilon>0$, and $\eta>0$ there exist some $\delta>0$ and $n_0\in\N$ such that for all $\mu_2\in{\cal N}$ with $d_{\scriptsize{\rm TV}}(\mu_1,\mu_2)\le\delta$,
\begin{equation}\label{Uniform WLLN - 1}
    \pr^{\mu_2}\Big[\,\Big|\frac{1}{n}\sum_{i=1}^n
    X_{0,i}-m_{\mu_2}\Big|\ge\eta\,\Big]\,\le\,\varepsilon\quad\mbox{ for all }n\ge n_0.
\end{equation}

(ii) If ${\cal N}\subset{\cal N}_1^1$ is even uniformly $\psi_1$-integrating, then for every $\varepsilon>0$ and $\eta>0$ there exists some $n_0\in\N$ such that
\begin{equation}\label{Uniform WLLN - 2}
    \sup_{\mu\in{\cal N}}\pr^{\mu}\Big[\,\Big|\frac{1}{n}\sum_{i=1}^n
    X_{0,i}-m_{\mu}\Big|\ge\eta\,\Big]\,\le\,\varepsilon\quad\mbox{ for all }n\ge n_0.
\end{equation}
\end{lemma}

\begin{proof}
Part (ii) is an immediate consequence of Chung's \cite{Chung1951} uniform (strong) law of large numbers. So it suffices to prove part (i). Fix $\mu_1 \in {\cal N}$, $\varepsilon \in (0,2)$ and $\eta>0$. For every $\ell\in\N$ let $X_{0,i}^\ell:=X_{0,i}\eins_{\{X_{0,i}\le\ell\}}$ be the $\ell$-truncation of $X_{0,i}$. Using the decomposition $X_{0,i}=X_{0,i}^\ell+X_{0,i}\eins_{\{X_{0,i}>\ell\}}$ and the triangle inequality, we obtain
\begin{eqnarray*}
    \pr^{\mu_2}\Big[\,\Big|\frac{1}{n}\sum_{i=1}^n X_{0,i}-m_{\mu_2}\Big|\ge\eta\Big]
    & \le & \pr^{\mu_2}\Big[\,\Big|\frac{1}{n}\sum_{i=1}^n X_{0,i}^\ell-\ex^{\mu_2}[X_{0,1}^\ell]\Big|\ge\eta/3\Big]\\
    & & +\,\pr^{\mu_2}\Big[\,\frac{1}{n}\sum_{i=1}^n X_{0,i}\eins_{\{X_{0,i}>\ell\}}\ge\eta/3\Big]\\
    & & +\,\pr^{\mu_2}\Big[\,\ex^{\mu_2}[X_{0,1}\eins_{\{X_{0,1}>\ell\}}]\ge\eta/3\Big]\\[1mm]
    & =: & S_1(\eta,n,\ell,\mu_2)+S_2(\eta,n,\ell,\mu_2)+S_3(\eta,\ell,\mu_2).
\end{eqnarray*}
By Markov's inequality $S_2(\eta,n,\ell,\mu_2)$ is bounded above by $3\eta^{-1}\ex^{\mu_2}[X_{0,1}\eins_{\{X_{0,1}> \ell\}}]$. The assumption on ${\cal N}$ yields that one can choose $\delta>0$ and $\ell_0=\ell_0(\varepsilon,\eta)\in\N$ such that
$$
    \mu_2\in{\cal N},\quad d_{\scriptsize{\rm TV}}(\mu_1,\mu_2)\le\delta\quad\Longrightarrow\quad \ex^{\mu_2}[X_{0,1}\eins_{\{X_{0,1} > \ell_0\}}]\,\le\, \eps \eta/6 < \eta/3.
$$
Hence
$$
    \mu_2\in{\cal N},\quad d_{\scriptsize{\rm TV}}(\mu_1,\mu_2)\le\delta\quad\Longrightarrow\quad S_2(\eta,n,\ell_0,\mu_2)+S_3(\eta,\ell_0,\mu_2) \leq \eps/2 + 0
$$
for all $n \in \N$.
By Chebychev's inequality, we further obtain (regardless of $\mu_2 \in {\cal N}_1^1$)
$$
   S_1(\eta,n,\ell_0,\mu_2)\,\le\,9\eta^{-2}\ell_0^2\,n^{-1} \leq \eps/2
$$
for all $n \geq n_0$ for some sufficiently large $n_0 \in \N$.
\end{proof}


\section{Uniform conditional weak consistency of the Lotka--Nagaev estimator}\label{Sec Uniform conditional weak consistency}

\begin{theorem}\label{uniform conditional weak consistency for Lotka}
(i) Let ${\cal N}\subset{\cal N}_1^1$ be a locally uniformly $\psi_1$-integrating set with $m_{\mu} > 1$ for all $\mu \in {\cal N}$. Then for every $\mu_1\in{\cal N}$, $\varepsilon>0$, and $\eta>0$ there exist some $\delta>0$ and $n_0\in\N$ such that for all $\mu_2\in{\cal N}$ with $d_{\scriptsize{\rm TV}}(\mu_1,\mu_2)\le\delta$,
\begin{equation}\label{uniform conditional weak consistency - eq - 10}
    \pr^{\mu_2}\big[|\widehat m_n-m_{\mu_2}|\ge\eta\,\big|Z_{n-1}>0\big]\,\le\,\varepsilon\quad\mbox{ for all }n\ge n_0.
\end{equation}

(ii) If ${\cal N}\subset{\cal N}_1^1$ is even uniformly $\psi_1$-integrating with $\inf_{\mu\in{\cal N}}m_\mu>1$, then for every $\varepsilon>0$ and $\eta>0$ there exists some $n_0\in\N$ such that
$$
    \sup_{\mu\in{\cal N}}\,\pr^\mu\big[|\widehat m_n-m_\mu|\ge\eta\,\big|Z_{n-1}>0\big]\,\leq\,\varepsilon\quad\mbox{ for all }n\ge n_0.
$$
\end{theorem}

\begin{proof}
We first prove part (i). Fix $\mu_1\in{\cal N}$, $\varepsilon>0$, and $\eta>0$. For every $\mu_2\in{\cal N}$ we have
\begin{eqnarray}
    \lefteqn{\pr^{\mu_2}\big[|\widehat m_n-m_{\mu_2}|\ge\eta\,\big|Z_{n-1}>0\big]}\nonumber\\
    & = & \sum_{k=1}^\infty\pr^{\mu_2}\big[|\widehat m_n-m_{\mu_2}|\ge\eta\,\big|Z_{n-1}=k\big]\,\pr^{\mu_2}[Z_{n-1}=k|Z_{n-1}>0]\nonumber\\
    & = & \sum_{k=1}^\infty\pr^{\mu_2}\big[|Z_n/k-m_{\mu_2}|\ge\eta\,\big|Z_{n-1}=k\big]\,\pr^{\mu_2}[Z_{n-1}=k|Z_{n-1}>0]\nonumber\\
    & = & \sum_{k=1}^\infty\pr^{\mu_2}\Big[\Big|\frac{1}{k}\sum_{i=1}^kX_{n-1,i}-m_{\mu_2}\Big|\ge\eta\Big]\,\pr^{\mu_2}[Z_{n-1}=k|Z_{n-1}>0].\label{uniform conditional weak consistency for Lotka - proof - 10}
\end{eqnarray}
By part (i) of Lemma \ref{Uniform WLLN}, we can find some $\delta>0$ and $k_0\in\N$ such that for all $\mu_2\in{\cal N}$ with $d_{\scriptsize{\rm TV}}(\mu_1,\mu_2)\le\delta$,
\begin{equation}\label{uniform conditional weak consistency for Lotka - proof - 20}
    \pr^{\mu_2}\Big[\Big|\frac{1}{k}\sum_{i=1}^kX_{n-1,i}-m_{\mu_2}\Big|\ge\eta\Big]\,\le\,\varepsilon/2\quad\mbox{ for all }k\ge k_0.
\end{equation}
From (\ref{uniform conditional weak consistency for Lotka - proof - 10}) and (\ref{uniform conditional weak consistency for Lotka - proof - 20}) we obtain that for all $\mu_2\in{\cal N}$ with $d_{\scriptsize{\rm TV}}(\mu_1,\mu_2)\le\delta$,
\begin{eqnarray}
    \lefteqn{\pr^{\mu_2}\big[|\widehat m_n-m_{\mu_2}|\ge\eta\,\big|Z_{n-1}>0\big]}\nonumber\\
    & \le & \varepsilon/2\,+\,\sum_{k=1}^{k_0}\pr^{\mu_2}\Big[\Big|\frac{1}{k}\sum_{i=1}^kX_{n-1,i}-m_{\mu_2}\Big|\ge\eta\Big]\,\pr^{\mu_2}[Z_{n-1}=k|Z_{n-1}>0]\nonumber\\
    & \le & \varepsilon/2\,+\,\sum_{k=1}^{k_0}\pr^{\mu_2}[Z_{n-1}=k|Z_{n-1}>0].\label{uniform conditional weak consistency for Lotka - proof - 40}
\end{eqnarray}
By part (i) of Lemma \ref{unifgrowth} we can find some $n_0\in\N$ (and decrease the $\delta>0$ chosen above if necessary) such that for all $\mu_2\in{\cal N}$ with $d_{\scriptsize{\rm TV}}(\mu_1,\mu_2)\le\delta$ and $n\ge n_0$,
\begin{equation}\label{uniform conditional weak consistency for Lotka - proof - 50}
    \pr^{\mu_2}[Z_{n-1}=k|Z_{n-1}>0]\,\le\,\varepsilon/(2k_0)\quad\mbox{ for all }k=1,\ldots,k_0.
\end{equation}
Now, (\ref{uniform conditional weak consistency for Lotka - proof - 40})--(\ref{uniform conditional weak consistency for Lotka - proof - 50}) yield that for all $\mu_2\in{\cal N}$ with $d_{\scriptsize{\rm TV}}(\mu_1,\mu_2)\le\delta$ and all $n\ge n_0$,
$$
    \pr^{\mu_2}\big[|\widehat m_n-m_{\mu_2}|\ge\eta\,\big|Z_{n-1}>0\big]\,\le\,\varepsilon.
$$
This implies (\ref{uniform conditional weak consistency - eq - 10}).

Part (ii) can be shown analogously. Use parts (ii) instead of (i) of Lemmas \ref{Uniform WLLN} and Lemma \ref{unifgrowth}, and remove the restriction $d_{\scriptsize{\rm TV}}(\mu_1,\mu_2)\le\delta$ everywhere.
\end{proof}


\section{Proof of Theorem \ref{asymptotic robustness of lotka}}\label{proof of main results}

Note that (uniform) robustness of $(\widehat m_n)$ on ${\cal N}\subset{\cal N}_1^1$ in the sense of Definition \ref{def robustness} holds if and only if $(\widehat m_n)$ is both (uniformly) asymptotically and (uniformly) finite sample robust on ${\cal N}$ in the following sense.

\begin{definition}\label{def robustness - asymp finite}
(i) The sequence $(\widehat m_n)$ is said to be asymptotically robust on ${\cal N}$ if for every $\mu_1\in{\cal N}$ and $\varepsilon>0$ there are some $\delta>0$ and $n_0\in\N$ such that
$$
    \mu_2\in{\cal N},\quad d_{\scriptsize{\rm TV}}(\mu_1,\mu_2)\le\delta\quad\Longrightarrow\quad \rho(\pr^{\mu_1}\circ \widehat m_n^{-1}\,,\,\pr^{\mu_2}\circ \widehat m_n^{-1})\le\varepsilon \quad\mbox{for all }n\ge n_0.
$$
It is said to be uniformly asymptotically robust on ${\cal N}$ if $\delta$ can be chosen independently of $\mu_1\in{\cal N}$.

(ii) The sequence $(\widehat m_n)$ is said to be finite sample robust on ${\cal N}$ if for every $\mu_1\in{\cal N}$, $n\in\N$, and $\varepsilon>0$ there  is some $\delta>0$ such that
\begin{equation}\label{def quali rob - finite sample}
    \mu_2\in{\cal N},\quad d_{\scriptsize{\rm TV}}(\mu_1,\mu_2)\le\delta\quad\Longrightarrow\quad \rho(\pr^{\mu_1}\circ \widehat m_n^{-1}\,,\,\pr^{\mu_2}\circ \widehat m_n^{-1})\le\varepsilon.
\end{equation}
It is said to be uniformly finite sample robust on ${\cal N}$ if $\delta$ can be chosen independently of $\mu_1\in{\cal N}$.
\end{definition}

The claim of Theorem \ref{asymptotic robustness of lotka} is an immediate consequence of Theorems \ref{hampel-huber generalized}, \ref{hampel-huber generalized - finite sample}, and \ref{hampel-huber converse} below. We first require the following lemma. Write $\pr_A^{\mu}[\,\cdot\,] := \pr^{\mu}[\,\cdot\,\vert A]$ for any $\mu \in {\cal N}_1^1$ and $A \in {\cal F}$.

\begin{lemma}\label{proof of hampel-huber generalized - lemma}
(i) Let ${\cal N}\subset{\cal N}_1^1$ be any set such that $m_\mu>1$ for all $\mu\in{\cal N}$. Then the sequence $(\widehat m_n)$ is asymptotically robust on ${\cal N}$ if and only if for every $\mu_1\in{\cal N}$ and $\varepsilon>0$ there are some $\delta>0$ and $n_0\in\N$ such that
\begin{eqnarray}
    \lefteqn{\mu_2\in{\cal N},\quad d_{\scriptsize{\rm TV}}(\mu_1,\mu_2)\le\delta}\nonumber\\
    & & \Longrightarrow\quad \rho(\pr_{\{Z_{n-1}>0\}}^{\mu_1}\circ \widehat m_n^{-1},\pr_{\{Z_{n-1}>0\}}^{\mu_2}\circ \widehat m_n^{-1})\le\varepsilon \quad\mbox{for all }n\ge n_0. \label{proof of hampel-huber generalized - lemma - 10}
\end{eqnarray}

(ii) Let ${\cal N}\subset{\cal N}_1^1$ be a uniformly $\psi_1$-integrating set with $\inf_{\mu\in{\cal N}}m_\mu>1$. Then the sequence $(\widehat m_n)$ is uniformly asymptotically robust on ${\cal N}$ if and only if for every $\varepsilon>0$ there are some $\delta>0$ and $n_0\in\N$ such that
\begin{eqnarray}
    \lefteqn{\mu_1,\mu_2\in{\cal N},\quad d_{\scriptsize{\rm TV}}(\mu_1,\mu_2)\le\delta}\nonumber\\
    & & \Longrightarrow\quad \rho(\pr_{\{Z_{n-1}>0\}}^{\mu_1}\circ \widehat m_n^{-1},\pr_{\{Z_{n-1}>0\}}^{\mu_2}\circ \widehat m_n^{-1})\le\varepsilon \quad\mbox{for all }n\ge n_0. \label{proof of hampel-huber generalized - lemma - 20}
\end{eqnarray}
\end{lemma}

\begin{proof}
We start by proving part (i). First assume that (\ref{proof of hampel-huber generalized - lemma - 10}) holds. By (\ref{proof of hampel-huber generalized - lemma - 10}) and part~(i) of Lemma \ref{proof of asymptotic robustness of lotka - lemma 1} we obtain that for every $\mu_1\in{\cal N}$ and $\varepsilon>0$ there are some $\delta>0$ and $n_0\in\N$ such that for every $n\ge n_0$, $\mu_2\in{\cal N}$ with $d_{\scriptsize{\rm TV}}(\mu_1,\mu_2)\le\delta$, and $A\in{\cal B}(\R_+)$,
\begin{eqnarray*}
    \lefteqn{\pr^{\mu_1}\circ \widehat m_n^{-1}[A]}\nonumber\\
    & = & \pr_{\{Z_{n-1}>0\}}^{\mu_1}\circ \widehat m_n^{-1}[A]\cdot\pr^{\mu_1}[Z_{n-1}>0]\,+\,\delta_0[A]\cdot\pr^{\mu_1}[Z_{n-1}=0]\\
    & \le & \big(\pr_{\{Z_{n-1}>0\}}^{\mu_2}\circ \widehat m_n^{-1} [A^\varepsilon]+\varepsilon\big)\cdot\big(\pr^{\mu_2}[Z_{n-1}>0]+\varepsilon\big)\,+\,\delta_0[A^\varepsilon]\cdot\big(\pr^{\mu_2}[Z_{n-1}=0]+\varepsilon\big)\\
    & \le & \pr_{\{Z_{n-1}>0\}}^{\mu_2}\circ \widehat m_n^{-1} [A^\varepsilon]\cdot\pr^{\mu_2}[Z_{n-1}>0]\,+\,\delta_0[A^\varepsilon]\cdot\pr^{\mu_2}[Z_{n-1}=0]\,+\,3\varepsilon\,+\,\varepsilon^2\\
    & = & \pr^{\mu_2}\circ \widehat m_n^{-1}[A^{\varepsilon}]\,+\,3\varepsilon\,+\,\varepsilon^2\\
    & \le & \pr^{\mu_2}\circ \widehat m_n^{-1}[A^{(3\varepsilon+\varepsilon^2)}]\,+\,(3\varepsilon+\varepsilon^2).
\end{eqnarray*}
Hence, we can find for every $\mu_1\in{\cal N}$ and $\widetilde\varepsilon>0$ some $\delta>0$ and $n_0\in\N$ such that
$$
    \mu_2\in{\cal N},\quad d_{\scriptsize{\rm TV}}(\mu_1,\mu_2)\le\delta\quad\Longrightarrow\quad \rho\big(\pr^{\mu_1}\circ\widehat m_n^{-1},\,\pr^{\mu_2}\circ\widehat m_n^{-1}\big)\le \widetilde\varepsilon \quad\mbox{for all }n\ge n_0.
$$
This means that $(\widehat m_n)$ is asymptotically robust on ${\cal N}$.

Now assume that the sequence $(\widehat m_n)$ is asymptotically robust on ${\cal N}$. It suffices to show that (\ref{proof of hampel-huber generalized - lemma - 10}) holds when the Prohorov metric $\rho$ is replaced by the bounded Lipschitz metric
$$
    \beta(\mu_1,\mu_2)\,:=\,\sup_{h\in{\rm BL}_1}\Big|\int h\,d\mu_1-\int h\,d\mu_2\Big|,
$$
where ${\rm BL}_1$ is the set of all functions $h:\R_+\rightarrow\R_+$ satisfying $\|h\|_{{\rm BL}}:=\|h\|_{{\rm L}}+\|h\|_{\infty}\le 1$ with $\|h\|_{{\rm L}}:=\sup_{x\not=y}|h(x)-h(y)|/|x-y|$ and $\|h\|_\infty:=\sup_x|h(x)|$; following the instructions on p.\,398 in \cite{Dudley2002} it can be easily shown that $\rho^2\le\frac{3}{2}\beta$. By the asymptotic robustness and part (i) of Lemma \ref{proof of asymptotic robustness of lotka - lemma 1} we obtain that for every $\mu_1\in{\cal N}$ and $\varepsilon \in (0,(1-q_{\mu_1})/2)$ there are some $\delta>0$ and $n_0\in\N$ such that for every $n\ge n_0$ and $\mu_2\in{\cal N}$ with $d_{\scriptsize{\rm TV}}(\mu_1,\mu_2)\le\delta$,
\begin{eqnarray}
    \lefteqn{\beta\big(\pr_{\{Z_{n-1}>0\}}^{\mu_1}\circ \widehat m_n^{-1},\pr_{\{Z_{n-1}>0\}}^{\mu_2}\circ \widehat m_n^{-1}\big)}\nonumber\\
    & \le & \sup_{h\in{\rm BL}_1}\Big|\frac{\int h\,d\pr^{\mu_1}\circ\widehat m_n^{-1}}{\pr^{\mu_1}[Z_{n-1}>0]}-\frac{\int h\,d\pr^{\mu_2}\circ\widehat m_n^{-1}}{\pr^{\mu_2}[Z_{n-1}>0]}\Big|\,+\,\Big|\frac{\pr^{\mu_1}[Z_{n-1}=0]}{\pr^{\mu_1}[Z_{n-1}>0]}-\frac{\pr^{\mu_2}[Z_{n-1}=0]}{\pr^{\mu_2}[Z_{n-1}>0]}\Big|\nonumber\\
    & \le & \frac{\sup_{h\in{\rm BL}_1}|\int h\,d\pr^{\mu_1}\circ\widehat m_n^{-1}-\int h\,d\pr^{\mu_2}\circ\widehat m_n^{-1}|}{\pr^{\mu_1}[Z_{n-1}>0]\cdot\pr^{\mu_2}[Z_{n-1}>0]}\,+\,\frac{|\pr^{\mu_2}[Z_{n-1}>0]-\pr^{\mu_1}[Z_{n-1}>0]|}{\pr^{\mu_1}[Z_{n-1}>0]\cdot\pr^{\mu_2}[Z_{n-1}>0]}\nonumber\\
    & & +\,\frac{|\pr^{\mu_1}[Z_{n-1}=0]-\pr^{\mu_2}[Z_{n-1}=0]|}{\pr^{\mu_1}[Z_{n-1}>0]\cdot\pr^{\mu_2}[Z_{n-1}>0]}\,+\,\frac{|\pr^{\mu_2}[Z_{n-1}>0]-\pr^{\mu_1}[Z_{n-1}>0]|}{\pr^{\mu_1}[Z_{n-1}>0]\cdot\pr^{\mu_2}[Z_{n-1}>0]}\nonumber\\
    & \le & 4\,\frac{\varepsilon}{\pr^{\mu_1}[Z_{n-1}>0]\cdot\pr^{\mu_2}[Z_{n-1}>0]} \ \le \ 4\,\frac{\varepsilon}{(1-q_{\mu_1})\cdot(1-q_{\mu_1}-\eps)}\,. \label{proof of hampel-huber generalized - lemma - 30}
\end{eqnarray}
This implies that (\ref{proof of hampel-huber generalized - lemma - 10}) holds (for the bounded Lipschitz metric).

Part (ii) can be shown analogously. Use part (ii) instead of (i) of
Lemma \ref{proof of asymptotic robustness of lotka - lemma 1}. Replace
the last bound in \eqref{proof of hampel-huber generalized - lemma -
  30} by $4\,\frac{\varepsilon}{(1-q_{\mu_1})\cdot(1-q_{\mu_2})}$,
which is less than or equal to $4 \eps/p^2$ by part (ii) of Lemma
\ref{uniformextinction} (further decreasing $\delta>0$ if necessary).
\end{proof}

\begin{theorem}\label{hampel-huber generalized}
(i) The sequence $(\widehat m_n)$ is asymptotically robust on any locally uniformly $\psi_1$-integrating set ${\cal N}\subset{\cal N}_1^1$ with $m_\mu>1$ for all $\mu\in{\cal N}$.

(ii) The sequence $(\widehat m_n)$ is uniformly asymptotically robust on any uniformly $\psi_1$-inte\-grating set ${\cal N}\subset{\cal N}_1^1$ with $\inf_{\mu\in{\cal N}}m_\mu>1$.
\end{theorem}

\begin{proof}
We first prove part (i). By Lemma \ref{proof of hampel-huber generalized - lemma}(i) it suffices to show that for every $\mu_1\in{\cal N}$ and $\varepsilon>0$ there are some $\delta>0$ and $n_0\in\N$ such that
\begin{eqnarray}
    \lefteqn{\mu_2\in{\cal N},\quad d_{\scriptsize{\rm TV}}(\mu_1,\mu_2)\le\delta}\nonumber\\
    & & \Longrightarrow\quad \rho(\pr_{\{Z_{n-1}>0\}}^{\mu_1}\circ \widehat m_n^{-1},\pr_{\{Z_{n-1}>0\}}^{\mu_2}\circ \widehat m_n^{-1})\le\varepsilon \quad\mbox{for all }n\ge n_0. \label{proof of hampel-huber generalized - eq - 1}
\end{eqnarray}
Fix $\mu_1 \in {\cal N}$ and $\eps > 0$. For every $\mu_2$ we have
\begin{eqnarray}
    \lefteqn{\rho(\pr_{\{Z_{n-1}>0\}}^{\mu_1}\circ \widehat m_n^{-1},\pr_{\{Z_{n-1}>0\}}^{\mu_2}\circ \widehat m_n^{-1})}\nonumber\\[1mm]
    &\le& \rho(\pr_{\{Z_{n-1}>0\}}^{\mu_1}\circ \widehat m_n^{-1},\delta_{m_{\mu_1}})\,+\,|m_{\mu_1}-m_{\mu_2}|\,+\,\rho(\delta_{m_{\mu_2}},\pr_{\{Z_{n-1}>0\}}^{\mu_2}\circ \widehat m_n^{-1}).\label{proof of hampel-huber generalized - eq - 10}
\end{eqnarray}

We start with the first and third summands in this bound. By part (i) of Theorem \ref{uniform conditional weak consistency for Lotka} we can find some $\delta>0$ and $n_0\in\N$ such that for all $n\ge n_0$ and $\mu_2\in{\cal N}$ with $d_{\scriptsize{\rm TV}}(\mu_1,\mu_2)\le\delta$,
\begin{equation*}\label{proof of hampel-huber generalized - eq - 6}
    \pr_{\{Z_{n-1}>0\}}^{\mu_2}\big[|\widehat m_n-m_{\mu_2}|\le\varepsilon/3\big]\,>\,1-\varepsilon/3.
\end{equation*}
Since $\{\widehat{m}_n \in A \} \subset \{ m_{\mu_2} \in A^{\eps/3} \} \cup \{|\widehat{m}_n - m_{\mu_2}| > \eps/3 \}$ for every $A \in \mathcal{B}(\R_{+})$, we obtain for every $A \in \mathcal{B}(\R_{+})$ that
\begin{equation*}
    \pr_{\{Z_{n-1}>0\}}^{\mu_2}\circ \widehat m_n^{-1}[A] \leq  \delta_{m_{\mu_2}}[A^{\eps/3}] + \eps/3,
\end{equation*}
and hence
\begin{equation*}
    \rho(\pr_{\{Z_{n-1}>0\}}^{\mu_2}\circ \widehat m_n^{-1},\delta_{m_{\mu_2}}) \leq \eps/3.
\end{equation*}

For the second summand on the right-hand side of \eqref{proof of hampel-huber generalized - eq - 10} we use the fact that $\mu\mapsto m_\mu$ is $(d_{\scriptsize{\rm TV}},|\cdot|)$-continuous at $\mu_1$, shown in Lemma \ref{continuity of m}(i). Decreasing $\delta>0$ above further if necessary, we obtain
\begin{equation}\label{proof of hampel-huber generalized - eq - 4}
     \mu_2\in{\cal N},\quad d_{\scriptsize{\rm TV}}(\mu_1,\mu_2)\le\delta\quad\Longrightarrow\quad |m_{\mu_1}-m_{\mu_2}|\le\varepsilon/3.
\end{equation}
This completes the proof of part (i).

Part (ii) can be shown analogously. Use parts (ii) instead of (i) of Lemma~\ref{proof of hampel-huber generalized - lemma}, Theorem~\ref{uniform conditional weak consistency for Lotka} and Lemma~\ref{continuity of m}. Note that a finite $\delta>0$ is only needed for the analogue of~\eqref{proof of hampel-huber generalized - eq - 4} (not before).
\end{proof}

\begin{theorem}\label{hampel-huber generalized - finite sample}
(i) The sequence $(\widehat m_n)$ is finite sample robust on ${\cal N}:={\cal N}_1^1$.

(ii) The sequence $(\widehat m_n)$ is uniformly finite sample robust on any uniformly $\psi_1$-inte\-grating set ${\cal N}\subset{\cal N}_1^1$.
\end{theorem}

\begin{proof}
We start by proving part (i). We have to show that for every $\mu_1\in{\cal N}$, $\varepsilon>0$, and $n\in\N$ there is some $\delta>0$ such that
\begin{equation}\label{proof of hampel-huber generalized - finite sample - eq - 10}
    \mu_2\in{\cal N},\quad d_{\scriptsize{\rm TV}}(\mu_1,\mu_2)\le\delta\quad\Longrightarrow\quad \rho(\pr^{\mu_1}\circ \widehat m_n^{-1},\pr^{\mu_2}\circ \widehat m_n^{-1})\le\varepsilon.
\end{equation}
By the simple direction in Strassen's theorem (e.g. Theorem~2.13 in \cite{HuberRonchetti2009}) the right-hand side in (\ref{proof of hampel-huber generalized - finite sample - eq - 10}) holds if we can find a probability measure $\nu = \nu_{\mu_1,\mu_2}$ on $(\R_+^2,{\cal B}(\R_+^2))$ such that
\begin{equation}\label{proof of hampel-huber generalized - finite sample - eq - 15}
    \nu\circ\pi_i^{-1}\,=\,\pr^{\mu_i}\circ\widehat m_n^{-1},\qquad i=1,2,
\end{equation}
(where $\pi_i:\R^2_+\to\R_+$ is the projection on the $i$th coordinate) and
\begin{equation}\label{proof of hampel-huber generalized - finite sample - eq - 20}
    \nu\big[\big\{(x_1,x_2)\in\R_+^2:\,|x_1-x_2|\le\varepsilon\big\}\big]\,\ge\,1-\varepsilon.
\end{equation}
Thus, for part (i) it suffices to show that for every $\mu_1\in{\cal N}$, $\varepsilon>0$, and $n\in\N$ there is some $\delta>0$ such that for every $\mu_2\in{\cal N}$ with $d_{\scriptsize{\rm TV}}(\mu_1,\mu_2)\le\delta$ one can find a probability measure $\nu$ on $(\R^2_+,{\cal B}(\R^2_+))$ satisfying (\ref{proof of hampel-huber generalized - finite sample - eq - 15})--(\ref{proof of hampel-huber generalized - finite sample - eq - 20}).

Let $\mu_1\in{\cal N}$, $\varepsilon>0$, and $n\in\N$ be fixed. By Lemma \ref{regularity of joint distributions}  we can find some $\delta>0$ such that
$$
    \mu_2\in{\cal N},\quad d_{\scriptsize{\rm TV}}(\mu_1,\mu_2)\le\delta\quad\Longrightarrow\quad d_{\scriptsize{\rm TV}}^{(2)}(\pr^{\mu_1}\circ(Z_{n-1},Z_n)^{-1},\pr^{\mu_2}\circ(Z_{n-1},Z_n)^{-1})\le\varepsilon.
$$
Together with Strassen's theorem this implies that for every $\mu_2\in{\cal N}$ with $d_{\scriptsize{\rm TV}}(\mu_1,\mu_2)\le\delta$ there is some probability measure $\widetilde\nu$ on $(\N_0^2\times\N_0^2,\mathfrak{P}(\N_0^2\times\N_0^2))$ such that
\begin{equation}\label{proof of hampel-huber generalized - finite sample - eq - 50}
    \widetilde\nu\circ\widetilde\pi_i^{-1}\,=\,\pr^{\mu_i}\circ(Z_{n-1},Z_n)^{-1},\qquad i=1,2,
\end{equation}
(where $\widetilde\pi_i:\N_0^2\times\N_0^2\to\N_0^2$ is the projection on the $i$th coordinate) and
\begin{equation}\label{proof of hampel-huber generalized - finite sample - eq - 60}
    \widetilde\nu\Big[\Big\{(z_{n-1}^1,z_n^1;\,z_{n-1}^2,z_n^2)\in\N_0^2\times\N_0^2:\|(z_{n-1}^1,z_n^1)-(z_{n-1}^2,z_n^2)\|\le\varepsilon\Big\}\Big]\,\ge\,1-\varepsilon,
\end{equation}
where $|| \cdot ||$ denotes the standard Euclidean norm.
Now, we set $\widehat{m}_n^{*}(Z_{n-1},Z_{n}):=Z_n/Z_{n-1}$ such that  $\widehat{m}_n = \widehat{m}_n^{*}(Z_{n-1},Z_{n})$, define
\begin{equation}\label{proof of hampel-huber generalized - finite sample - eq - 70}
    \nu\,:=\,\widetilde\nu\circ(\widehat m_n^{*}\circ\widetilde\pi_1,\widehat m_n^{*}\circ\widetilde\pi_2)^{-1}.
\end{equation}
From (\ref{proof of hampel-huber generalized - finite sample - eq - 50}) we obtain for $i=1,2$
\begin{eqnarray*}
    \nu\circ\pi_i^{-1}
    & = & (\widetilde\nu\circ(\widehat m_n^{*}\circ\widetilde\pi_1,\widehat m_n^{*}\circ\widetilde\pi_2)^{-1})\circ\pi_i^{-1}\\
    & = & \widetilde\nu\circ(\pi_i\circ(\widehat m_n^{*}\circ\widetilde\pi_1,\widehat m_n^{*}\circ\widetilde\pi_2))^{-1}\\
    & = & \widetilde\nu\circ(\widehat m_n^{*}\circ\widetilde\pi_i)^{-1}\\
    & = & (\widetilde\nu\circ\widetilde \pi_i^{-1})\circ\widehat m_n^{*}{}^{-1}\\
    & = & (\pr^{\mu_i}\circ(Z_{n-1},Z_n)^{-1})\circ\widehat m_n^{*}{}^{-1}\\
    & = & \pr^{\mu_i}\circ(\widehat m_n^{*}\circ(Z_{n-1},Z_n))^{-1}\\
    & = & \pr^{\mu_i}\circ\widehat m_n^{-1}.
\end{eqnarray*}
That is, (\ref{proof of hampel-huber generalized - finite sample - eq - 15}) holds for $\nu$ defined in (\ref{proof of hampel-huber generalized - finite sample - eq - 70}). Further, if $\|(z_{n-1}^1,z_n^1)-(z_{n-1}^2,z_n^2)\|<1$, then $(z_{n-1}^1,z_n^1)=(z_{n-1}^2,z_n^2)$ and so $\widehat m_n(z_{n-1}^1,z_n^1)=\widehat m_n(z_{n-1}^2,z_n^2)$. Thus, assuming without loss of generality $0<\varepsilon<1$, we obtain
\begin{eqnarray*}
    \lefteqn{\nu\big[\big\{(x_1,x_2)\in\R_+^2:\,|x_1-x_2|>\varepsilon\big\}\big]}\\
    & = & \widetilde\nu\big[\big\{(z_{n-1}^1,z_n^1;\,z_{n-1}^2,z_n^2)\in\N_0^2\times\N_0^2:\,|\widehat m_n^{*}(z_{n-1}^1,z_n^1)-\widehat m_n^{*}(z_{n-1}^2,z_n^2)|>\varepsilon\big\}\big]\\
    & \le & \widetilde\nu\big[\big\{(z_{n-1}^1,z_n^1;\,z_{n-1}^2,z_n^2)\in\N_0^2\times\N_0^2:\,(z_{n-1}^1,z_n^1) \neq (z_{n-1}^2,z_n^2)\big]\\
    & < & \varepsilon,
\end{eqnarray*}
where the last step is ensured by (\ref{proof of hampel-huber generalized - finite sample - eq - 60}). That is, we also have (\ref{proof of hampel-huber generalized - finite sample - eq - 20}) for $\nu$ defined in (\ref{proof of hampel-huber generalized - finite sample - eq - 70}). This completes the proof of part (i).

Part (ii) can be shown analogously. Take into account that, under the stronger assumption on ${\cal N}$, Lemma \ref{regularity of joint distributions} and part (ii) of Lemma \ref{continuity of m} imply that the mapping ${\cal N}_1^1\to ({\cal N}_1^{1})^{2}$, $\mu\mapsto \pr^\mu\circ(Z_{n-1},Z_n)^{-1}$ is {\em uniformly} $(d_{\scriptsize{\rm TV}},d_{\scriptsize{\rm TV}}^{(2)})$-continuous.
\end{proof}

\begin{theorem}\label{hampel-huber converse}
Let ${\cal N}\subset{\cal N}_1^1$ such that $m_\mu>1$ for all $\mu\in{\cal N}$, and assume that there exists some $\mu_1\in{\cal N}$ such that the mapping ${\cal N}\ni\mu\mapsto m_\mu$ is not $(d_{\scriptsize{\rm TV}},|\cdot|)$-continuous at~$\mu_1$. Then the sequence $(\widehat m_n)$ is not asymptotically robust on ${\cal N}$.
\end{theorem}

\begin{proof}
Suppose that the sequence $(\widehat m_n)$ is asymptotically robust on ${\cal N}$. In view of the identity $\min\{1;|m_{\mu_1}-m_{\mu_2}|\}=\rho(\delta_{m_{\mu_1}},\delta_{m_{\mu_2}})$, we have for every $\mu_2\in{\cal N}$,
\begin{eqnarray*}
    \lefteqn{\min\{1;|m_{\mu_1}-m_{\mu_2}|\}}\\
    & \le & \rho\big(\pr_{\{Z_{n-1}>0\}}^{\mu_1}\circ \widehat m_n^{-1},\,\pr_{\{Z_{n-1}>0\}}^{\mu_2}\circ \widehat m_n^{-1}\big)\,+\,\sum_{i=1}^2\rho(\delta_{m_{\mu_i}},\,\pr_{\{Z_{n-1}>0\}}^{\mu_i}\circ \widehat m_n^{-1})\\
    & =: & S_0(n,\mu_1,\mu_2)\,+\,\sum_{i=1}^2S_{i}(n,\mu_i).
\end{eqnarray*}
Let $\varepsilon>0$ be fixed. Recall that $\rho$ metrizes the weak topology. Thus, using Theorem~\ref{uniform conditional weak consistency for Lotka}(ii) with ${\cal N} = \{\mu_i\}$, we can find some $n_1\in\N$ such that
$$
    \sum_{i=1}^2S_{i}(n,\mu_i)\,\le\,\varepsilon/2\quad\mbox{ for all }n\ge n_1.
$$
By the asymptotic robustness of $(\widehat m_n)$ and part (i) of Lemma \ref{proof of hampel-huber generalized - lemma}, we can also find some $\delta>0$ and $n_0\ge n_1$ such that
$$
    \mu_2\in{\cal N},\quad d_{\scriptsize{\rm TV}}(\mu_1,\mu_2)\le\delta\quad\Longrightarrow\quad S_0(n,\mu_1,\mu_2)\,\le\,\varepsilon/2\quad\mbox{ for all }n\ge n_0.
$$
Thus, the mapping $\mu\mapsto m_{\mu}$ is $(d_{\scriptsize{\rm TV}},|\cdot|)$-continuous at $\mu_1$. This contradicts the assumption.
\end{proof}


\section{Extension to general initial states}\label{Extension to general initial states}

In this section, we outline modifications in the arguments that show that our main result, Theorem \ref{asymptotic robustness of lotka}, is true when we start the process with a population of general size $z_0.$ Note that in this case, we can decompose the process $(Z_n)$ into $z_0$ independent processes
$(Z_n^{(i)})$ started with $1$ individual for $i=1,\dots, z_0$ such that
\begin{equation}
\label{general z0 decomposition}
Z_n= Z_n^{(1)}+ \dots + Z_n^{(z_0)}.
\end{equation}
In order to avoid confusion we write $\pr^{z_0,\mu}$ for the probability measure under which $(Z_n)$ started in $z_0$ with offspring distribution $\mu$ evolves.
Denoting by $q_\mu^{(z_0)}$ the extinction probability of $(Z_n),$ it is immediate that $q_{\mu}^{(z_0)}=q_{\mu}^{z_0}\leq q_{\mu}$ for all $z_0 \in \N.$
We will show that Theorems \ref{hampel-huber generalized} and  \ref{hampel-huber generalized - finite sample} hold also for $(Z_n)$ started in a general $z_0$ such that Theorem \ref{asymptotic robustness of lotka} follows.

Theorem \ref{hampel-huber generalized} uses Lemma \ref{proof of hampel-huber generalized - lemma} whose proof works in the same way as before: We simply have to note that Lemma \ref{proof of asymptotic robustness of lotka - lemma 1} still holds due to the inequality
\begin{eqnarray*}
|\pr^{z_0,\mu_1}[Z_{n}=0]-\pr^{z_0,\mu_2}[Z_{n}=0]|&=& |\pr^{\mu_1}[Z_{n}=0]^{z_0}-\pr^{\mu_2}[Z_{n}=0]^{z_0}|\\
&\leq& z_0  |\pr^{\mu_1}[Z_{n}=0]-\pr^{\mu_2}[Z_{n}=0]|
\end{eqnarray*}
and replace $q_{\mu_i}$ by  $q_{\mu_i}^{z_0}$ for $i=1,2$ in the argument.

The other result that is needed in Theorem \ref{hampel-huber generalized} is Theorem \ref{uniform conditional weak consistency for Lotka}. The proof of the latter still applies as long as Lemma \ref{unifgrowth} holds. The modifications here are the following: According to (\ref{from alive at n to overall survival}), replacing $q_{\mu_2}$ by $q_{\mu_2}^{z_0}$ we need to bound  $\pr^{z_0,\mu_2}[ Z_n =k| A]$ and $\pr^{z_0,\mu_2}[ B | Z_n>0]$ where $A$ is the event of survival and $B$ that of extinction of $(Z_n)$. For the former we use that for all $z_0 \in \N$,
\begin{equation*}
    \pr^{z_0,\mu_2}[ Z_n =k| A]\, \leq\, \pr^{z_0,\mu_2}[ Z_n \leq k| A]\, \leq\, \pr^{\mu_2}[ Z_n \leq k| A]
\end{equation*}
in (\ref{binomial bound}).
Replacing again $q_{\mu_2}$ by $q_{\mu_2}^{z_0}$ we see by Bayes Formula that  for the latter it suffices to consider
\begin{equation*}
    \pr^{z_0,\mu_2}[Z_n>0 | B ]\, \leq\, \sum_{i=1}^{z_0} \pr^{z_0,\mu_2}[Z_n^{(i)}>0 | B ]\, \leq\, z_0 \pr^{\mu_2}[ Z_n^{(i)} >0| B_i],
\end{equation*}
where $B_i$ denotes the event of extinction of $Z_n^{(i)}.$ The last probability is bounded appropriately in the proof of Lemma \ref{unifgrowth}.

Having established the validity of Theorem \ref{hampel-huber generalized} we turn to Theorem \ref{hampel-huber generalized - finite sample}. Here, the essential ingredient is the analogous version of Lemma \ref{regularity of joint distributions}. However, it is easy to see that
(\ref{regularity of joint distributions - eq}) holds with $C_n(\mu_1,\mu_2)$ replaced by $z_0 C_n(\mu_1,\mu_2)$: Namely, note that
due to (\ref{general z0 decomposition})
\begin{eqnarray*}
 & & d_{\scriptsize{\rm TV}}^{(n)}\big(\pr^{z_0,\mu_1}\circ(Z_1,\ldots,Z_n)^{-1},\pr^{z_0,\mu_2}\circ(Z_1,\ldots,Z_n)^{-1}\big)\\
 &=& \frac{1}{2} 
\sum\big(\pr^{z_0,\mu_1}[Z_1^{(1)}=z_1^{(1)},\ldots,Z_n^{(1)}=z_n^{(1)}, \ldots, Z_1^{(z_0)}=z_1^{(z_0)},\ldots, Z_n^{(z_0)}=z_n^{(z_0)}]\\
 & & \phantom{AAAA}
 -  \pr^{z_0,\mu_2}[Z_1^{(1)}=z_1^{(1)},\ldots,Z_n^{(1)}=z_n^{(1)}, \ldots, Z_1^{(z_0)}=z_1^{(z_0)},\ldots, Z_n^{(z_0)}=z_n^{(z_0)}]  \big),
\end{eqnarray*}
where the sum ranges over all $z_1^{(1)},\ldots,z_n^{(1)},\cdots,z_1^{(z_0)}, \dots, z_n^{(z_0)} \in \N_0$. But due to the independence of $(Z^{(1)}_n)$ to   $(Z^{(z_0)}_n)$ we have
\begin{eqnarray*}
& &\pr^{z_0,\mu_1}[Z_1^{(1)}=z_1^{(1)},\ldots,Z_n^{(1)}=z_n^{(1)}, \ldots, Z_1^{(z_0)}=z_1^{(z_0)},\ldots, Z_n^{(z_0)}=z_n^{(z_0)}]\\
&=& \prod_{i=1}^{z_0} \pr^{\mu_1}[Z_1^{(1)}=z_1^{(1)},\ldots,Z_n^{(1)}=z_n^{(1)}]
\end{eqnarray*}
so that it follows with (\ref{regularity of joint distributions - proof - 2}) that
\begin{eqnarray*}
 & &d_{\scriptsize{\rm TV}}^{(n)}\big(\pr^{z_0,\mu_1}\circ(Z_1,\ldots,Z_n)^{-1},\pr^{z_0,\mu_2}\circ(Z_1,\ldots,Z_n)^{-1}\big)\\
&\leq& z_0 d_{\scriptsize{\rm TV}}^{(n)}\big(\pr^{\mu_1}\circ(Z_1,\ldots,Z_n)^{-1},\pr^{\mu_2}\circ(Z_1,\ldots,Z_n)^{-1}\big).
\end{eqnarray*}
The conclusion follows now with the original version of Lemma \ref{regularity of joint distributions}. This completes the proof of Theorem \ref{hampel-huber generalized - finite sample} and thus also of Theorem \ref{asymptotic robustness of lotka}.

%




\begin{thebibliography}{0}
    \bibitem{AsmussenHering1983} Asmussen, S. and Hering, H. (1983). {\it Branching processes}. Birkhäuser, Boston.
    \bibitem{AthreyaNey1972} Athreya, K.B. and Ney, P.E. (1972). {\it Branching processes}. Dover, New York.
    \bibitem{Chung1951} Chung, K.L. (1951). The strong law of large numbers. {\it Proceedings of the Second Berkeley Symposium on Mathematical Statistics and Probability}, University of California Press, Berkeley and Los Angeles, 341--352.
    \bibitem{Contetal2010} Cont, R., Deguest, R. and Scandolo, G. (2010) {\it Robustness and sensitivity analysis of risk measurement procedures.} Quantitative Finance, 10, 593--606.
    \bibitem{Cuevas1988} Cuevas, A. (1988). Qualitative robustness in abstract inference. {\it Journal of Statistical Planning and Inference}, 18, 277--289.
    \bibitem{Dion1974} Dion, J.-P. (1974). Estimation of the mean and the initial probabilities of a branching process. {\it Journal of Applied Probability}, 11, 687--694.
    \bibitem{DionKeiding1978} Dion, J.-P. and Keiding, N. (1978). Statistical inference in branching processes. {\it In: Branching processes}, Joffe, A. and Ney, P., editors. Advances in Probability and Related Topics 5, Dekker, New York.
    \bibitem{Dudley2002} Dudley, R.M. (2002). {\it Real analysis and probability}. Cambridge University Press, Cambridge.
    \bibitem{Feigin1977} Feigin, P.D. (1977). A note on maximum likelihood estimation for simple branching processes. {\it  Australian Journal of Statistics}, 19, 152--154.
    \bibitem{FoellmerSchied2011} F\"ollmer, H. and Schied, A. (2011). {\it Stochastic finance. An introduction in discrete time}. De Gruyter, Berlin.
    \bibitem{Hampel1971} Hampel, F.R. (1971). A general qualitative definition of robustness. {\it Annals of Mathematical Statistics}, 42, 1887--1896.
    \bibitem{Hampeletal1986} Hampel, F.R., Ronchetti, E.M., Rousseeuw, P.J. and Stahel, W.A. (1986). {\it Robust statistics -- the approach based on influence functions}, Wiley, New York.
    \bibitem{Harris1948} Harris, T.E. (1948). Branching processes. {\it Annals of Mathematical Statistics}, 19, 474--494.
    \bibitem{Heyde1970} Heyde, C.C. (1970). Extension of a result of Seneta for the super-critical Galton--Watson process. {\it Annals of Mathematical Statistics}, 41, 739--742.
    \bibitem{HuberRonchetti2009} Huber, P.J. and Ronchetti, E.M. (2009). {\it Robust statistics}. Wiley, New York.
    \bibitem{KeidingLauritzen1978} Keiding, N. and Lauritzen, S. (1978). Marginal maximum likelihood estimates and estimation of the offspring mean in a branching process. {\it Scandinavian Journal of Statistics}, 5, 106--110.
    \bibitem{Kraetschmeretal2014} Kr\"atschmer, V., Schied, A. and Z\"ahle, H. (2014). Comparative and qualitative robustness for law-invariant risk measures. {\it Finance and Stochastics}, 18, 271--295.
    \bibitem{Lotka1939} Lotka, A. (1939). Theorie analytique des associations biologiques. {\it Actualités Sci. Ind.}, 780, 123--136.
    \bibitem{Meyer1966} Meyer, P. (1966). {\it Probability and Potentials}, Blaisdell Publishing Co., Waltham.
    \bibitem{MitovYanev2009} Mitov, K. and Yanev, N. (2009). Branching stochastic processes: Regulation, regeneration, estimation, applications. {\it Pliska Studia Mathematica Bulgarica}, 19, 5--58.
    \bibitem{Nagaev1967} Nagaev, A.V. (1967). On estimating the expected number of direct descendants of a particle in a branching process. {\it Theory of Probability and its Applications}, 12, 314--320.
    \bibitem{Stoimenovaetal2004} Stoimenova, V., Atanasov, D. and Yanev, N. (2004). Simulation and robust modifications of estimates in branching processes. {\it Pliska Studia Mathematica Bulgarica}, 16, 259--271.
    \bibitem{VanderVaart1998} Van der Vaart, A.W. (1998). {\it Asymptotic statistics}. Cambridge University Press, Cambridge.
    \bibitem{Zaehle2014b} Z{\"a}hle, H. (2014). A definition of qualitative robustness for general point estimators, and examples. {\it Submitted} (ArXiv:1406.7711).
\end{thebibliography}
\end{document}